\documentclass[a4paper,11pt]{article}
\usepackage{MyStyle}

\usepackage{algorithm}
\floatname{algorithm}{}
\usepackage{algorithmic}

\usepackage{hyperref}

\allowdisplaybreaks

\newtheorem{theorem}{Theorem}[section]

\newtheorem*{definition*}{Definition}

\newtheorem{proposition}[theorem]{Proposition}
\newtheorem{lemma}[theorem]{Lemma}
\newtheorem{remark}[theorem]{Remark}
\newtheorem*{remark*}{Remark}

\newtheorem*{remarks*}{Remarks}

\newtheorem{ass}[theorem]{Assumption}
\newtheorem*{notation*}{Notation}

\newtheorem*{ex*}{Example}

\newtheorem*{exs*}{Examples}

\newtheorem*{app*}{Application}
\newtheorem{conjecture*}{Conjecture}

\def\ts{\thinspace}

\newcommand{\methodUA}{New Method}

\title{
A uniformly accurate scheme for the numerical integration of penalized Langevin dynamics
}

\author{
Adrien Laurent\textsuperscript{1}
}

\begin{document}
\footnotetext[1]{
Department of Mathematics, University of Bergen, Norway. Adrien.Laurent@uib.no.}

\maketitle

\begin{abstract}
In molecular dynamics, penalized overdamped Langevin dynamics are used to model the motion of a set of particles that follow constraints up to a parameter~$\varepsilon$.
The most used schemes for simulating these dynamics are the Euler integrator in~$\R^d$ and the constrained Euler integrator. Both have weak order one of accuracy, but work properly only in specific regimes depending on the size of the parameter~$\varepsilon$.
We propose in this paper a new consistent method with an accuracy independent of~$\varepsilon$ for solving penalized dynamics on a manifold of any dimension. Moreover, this method converges to the constrained Euler scheme when~$\varepsilon$ goes to zero.
The numerical experiments confirm the theoretical findings, in the context of weak convergence and for the invariant measure, on a torus and on the orthogonal group in high dimension and high codimension.

\smallskip

\smallskip

 \end{abstract}

\section{Introduction}

In molecular dynamics, the overdamped Langevin equation in~$\R^d$ is often used for modeling the behavior of a large set of particles in a high friction regime. It is given by
\begin{equation}
\label{equation:Langevin}
dX(t)=f(X(t))dt +\sigma dW(t),\quad X(0)=X_0,
\end{equation}
where~$f$ is a smooth Lipschitz function (typically of the form~$f=-\nabla V$ for~$V$ a smooth potential),~$\sigma>0$ is a constant scalar, and~$W$ is a standard~$d$-dimensional Brownian motion in~$\R^d$ on a probability space equipped with a filtration~$(\FF_t)$ and fulfilling the usual assumptions.
If the particles are subject to smooth constraints~$\zeta\colon\R^d\rightarrow\R^q$, such as strong covalent bonds between atoms or fixed angles in molecules, the dynamics follow the constrained overdamped Langevin equation
\begin{equation}
\label{equation:projected_Langevin}
dX^0(t)=\Pi_\MM(X^0(t)) f(X^0(t))dt +\sigma \Pi_\MM(X^0(t)) \circ dW(t),\quad X^0(0)=X_0\in \MM,
\end{equation}
where the solution lies on the manifold~$\MM=\{x\in\R^d, \zeta(x)=0\}$ with codimension~$q$ thanks to~$\Pi_\MM\colon\R^d\rightarrow \R^{d\times d}$, the orthogonal projection on the tangent bundle of the manifold~$\MM$.

In physical applications, constrained systems are often used as a limit model for stiff equations. For instance, in the dynamics of a diatomic molecule, the distance between the two atoms oscillates around an average length, called the bond length (see, for instance,~\cite[Sect.\ts 1.2.1]{Lelievre10fec} on the interactions of particles). One can work with a simpler constrained dynamics where the distance between the atoms is fixed as a constraint, or with the original (possibly stiff) dynamics in~$\R^d$. We refer the reader to~\cite{Plechac09imm,Lelievre10fec}, and references therein, for discussions on the use of constraints and penalizations in molecular dynamics.
For overdamped Langevin dynamics~\eqref{equation:Langevin}, choosing a function~$\widetilde{f}=-\nabla \widetilde{V}$ with the potential
$$\widetilde{V}=V-\frac{\sigma^2}{4}\ln(\det(G))+\frac{1}{2\varepsilon}\abs{\zeta}^2$$
gives penalized Langevin dynamics~\cite{Ciccotti08pod} of the form
\begin{equation}
\label{equation:modified_Langevin_epsilon}
dY^\varepsilon(t)=f(Y^\varepsilon(t)) dt +\sigma dW(t) +\frac{\sigma^2}{4}\nabla\ln(\det(G))(Y^\varepsilon(t))dt -\frac{1}{\varepsilon}(g \zeta)(Y^\varepsilon(t))dt,
\end{equation}
where we fix~$Y^\varepsilon(0)=X_0$, the parameter~$\varepsilon>0$ is fixed with arbitrary size,~$f=-\nabla V$,~$g=\nabla \zeta\colon\R^d\rightarrow\R^{d\times q}$, and~$G=g^T g\colon \R^d\rightarrow\R^{q\times q}$ is the Gram matrix.
It was shown in~\cite[Appx.\ts C]{Ciccotti08pod} that the solution~$Y^\varepsilon$ of~\eqref{equation:modified_Langevin_epsilon} converges strongly to the solution~$X^0$ of the constrained dynamics~\eqref{equation:projected_Langevin} if~$X_0\in\MM$.
The additional term~$\frac{\sigma^2}{4}\nabla\ln(\det(G))$ is a correction term (called the Fixman correction) that is needed to obtain the convergence to the constrained dynamics~\eqref{equation:projected_Langevin} (see~\cite[Sect.\ts 3.2.3.4]{Lelievre10fec} and references therein).
Thus, for~$\varepsilon$ small, the trajectory of the solution of~\eqref{equation:modified_Langevin_epsilon} lies in the vicinity of the manifold~$\MM$. This penalization can also appear naturally when simulating Langevin dynamics with a stiff potential (see, for instance,~\cite[Sect.\ts 5.1]{Vilmart15pif}). One is then interested in numerical schemes that are robust with respect to the parameter~$\varepsilon$ and that lie on the manifold~$\MM$ in the limit~$\varepsilon\rightarrow 0$.
In this paper, we study the following similar penalized dynamics in~$\R^d$ to simulate trajectories in a vicinity of the manifold~$\MM$:
\begin{align}
\label{equation:Langevin_epsilon}
dX^\varepsilon(t)&=f(X^\varepsilon(t)) dt +\sigma dW(t) +\frac{\sigma^2}{4}\nabla\ln(\det(G))(X^\varepsilon(t))dt 
-\frac{1}{\varepsilon}(g G^{-1}\zeta)(X^\varepsilon(t))dt,
\end{align}
where~$X^\varepsilon(0)=X_0$. It is a simpler version of~\eqref{equation:modified_Langevin_epsilon} that also evolves in a vicinity of the manifold~$\MM$ in the limit~$\varepsilon\rightarrow 0$.
One result of this paper is the strong convergence of the solution~$X^\varepsilon$ of~\eqref{equation:Langevin_epsilon} to the solution~$X^0$ of~\eqref{equation:projected_Langevin} if~$X_0\in\MM$.
We mention that in the deterministic setting, that is, when~$\sigma=0$, equation~\eqref{equation:Langevin_epsilon} is a singular perturbation problem, and it converges to a differential algebraic equation (DAE) of index two in the limit~$\varepsilon\rightarrow 0$ (see~\cite[Chaps.\ts VI-VII]{Hairer10sod}).
We propose in this article a method that is robust with respect to the parameter~$\varepsilon$ for solving equations of the form~\eqref{equation:Langevin_epsilon}, and we leave the creation of robust integrators for solving~\eqref{equation:modified_Langevin_epsilon} for future work for the sake of clarity.

There are different ways to approximate the solution of the dynamics~\eqref{equation:Langevin_epsilon}.
A strong approximation focuses on approximating the realization of a single trajectory of~\eqref{equation:Langevin_epsilon} for a given realization of the Wiener process~$W$.
A weak approximation approximates the average of functionals of the solution at a fixed time~$T$, that is, quantities of the form~$\E[\phi(X^\varepsilon(T))]$ for~$\phi$ a smooth test function.
In addition, under growth and smoothness assumptions on the vector fields in~\eqref{equation:Langevin_epsilon} (see, for instance,~\cite{Hasminskii80sso}), the dynamics~\eqref{equation:Langevin_epsilon} naturally satisfy an ergodicity property; that is, there exists a unique invariant measure~$d\mu_\infty^\varepsilon$ in~$\R^d$ that has a density~$\rho_\infty^\varepsilon$ with respect to the Lebesgue measure, such that for all test functions~$\phi$,
$$
\lim_{T\to\infty}\frac{1}{T}\int_0^T \phi(X(t)) dt= \int_{\R^d} \phi(x) d\mu_\infty^\varepsilon(x)\quad \text{almost surely}.
$$
An approximation for the invariant measure focuses on approximating the average of a functional in the stationary state, that is, the quantity~$\int_{\R^d} \phi(x) d\mu_\infty^\varepsilon(x)$.
This is a computational challenge when the dimension~$d$ is high, which is the case in the context of molecular dynamics where the dimension is proportional to the number of particles, as a standard quadrature formula becomes prohibitively expensive in high dimension.
We emphasize that the invariant measure~$\mu_\infty^\varepsilon$ becomes singular with respect to the Lebesgue measure on~$\R^d$ in the limit~$\varepsilon\rightarrow 0$, and tends weakly as~$\varepsilon\rightarrow 0$ to~$d\mu_\infty^0$, a measure that is absolutely continuous to~$d\sigma_\MM$, the canonical measure on~$\MM$ induced by the Euclidean metric of~$\R^d$.
In this paper, we propose weak convergence results for a new uniformly accurate integrator for solving~\eqref{equation:Langevin_epsilon}, and numerical experiments in the weak context and for the invariant measure, as we recall that a scheme of weak order~$r$ automatically has order~$p\geq r$ for the invariant measure (see, for instance,~\cite{Mattingly02efs}).

The most used discretization for solving~\eqref{equation:Langevin_epsilon} is the explicit Euler integrator in~$\R^d$ (see~\cite{Ciccotti05bms,Lelievre08adc,Lelievre10fec,Lelievre12ldw}, for instance),
\begin{equation}
\label{equation:Euler_EEE}
X_{n+1}=X_n+\sqrt{h}\sigma \xi+h f(X_n) +h \frac{\sigma^2}{4}\nabla\ln(\det(G))(X_n) -\frac{h}{\varepsilon}(g G^{-1}\zeta)(X_n).
\end{equation}
This integrator has weak order one of accuracy, but it faces some severe stepsize restriction due to its instability, typically of the form~$h\ll \varepsilon$, in order to be accurate in the regime~$\varepsilon\rightarrow 0$.
Since the solution~$X^\varepsilon(t)$ of~\eqref{equation:Langevin_epsilon} converges to the solution~$X^0(t)$ of~\eqref{equation:projected_Langevin} when~$\varepsilon\rightarrow 0$, one can use integrators for the limit equation~\eqref{equation:projected_Langevin} and apply them to solve the original problem~\eqref{equation:Langevin_epsilon} when~$\varepsilon$ is close to zero.
Indeed, one can prove that the solution $X^\varepsilon(t)$ of~\eqref{equation:Langevin_epsilon} stays at distance $\OO(\sqrt{\varepsilon})$ of the constrained solution $X^0(t)$ of~\eqref{equation:projected_Langevin} (see Theorem \ref{theorem:CV_penalized_constrained}).
Thus, if the timestep of the integrator is small enough and satisfies~$\varepsilon\ll h$, then this integrator is consistent for solving~\eqref{equation:Langevin_epsilon}.
The alternative for the discretization of~\eqref{equation:projected_Langevin} on the manifold of the explicit Euler scheme~\eqref{equation:Euler_EEE} is the constrained Euler scheme
\begin{equation}
\label{equation:Euler_Explicit_constrained}
X^0_{n+1}=X^0_n+\sqrt{h}\sigma \xi +hf(X^0_n) + g(X^0_n)\lambda^0_{n+1}, \quad \zeta(X^0_{n+1})=0,
\end{equation}
where~$\lambda^0_{n+1}\in \R^q$ acts as a Lagrange multiplier, and is entirely determined by the constraint~$\zeta(X^0_{n+1})=0$. This integrator has weak order one for solving~\eqref{equation:projected_Langevin} (see~\cite[Sect.\ts 3.2.4]{Lelievre10fec}) and it lies on the manifold~$\MM$. It is a consistent approximation of~\eqref{equation:Langevin_epsilon} if~$\varepsilon$ is close to zero and~$\varepsilon\ll h$.
This integrator is, however, not appropriate for solving~\eqref{equation:Langevin_epsilon} if the size of~$\varepsilon$ is of the order of one, since the exact solution does not evolve in a neighborhood of the manifold in this regime.
We mention a few other techniques to integrate numerically~\eqref{equation:Langevin_epsilon} or~\eqref{equation:projected_Langevin}.
In~\cite{Laurent20eab,Laurent21ocf}, high order Runge-Kutta methods are proposed for sampling the invariant measure in~$\R^d$ and on manifolds.
The paper~\cite{Lelievre19hmc} presents a constrained integrator based on the RATTLE scheme (see~\cite{Ryckaert77nio,Andersen83rav,Hairer10sod}) in the context of the underdamped Langevin dynamics.
Some of the previously cited discretizations can be combined with Metropolis-Hastings rejection procedures~\cite{Metropolis53eos,Hastings70mcs}, as done, for instance, in~\cite{Girolami11rml,Brubaker12afo,Lelievre12ldw,Zappa18mco,Lelievre19hmc}.
As for the Euler integrators~\eqref{equation:Euler_EEE}-\eqref{equation:Euler_Explicit_constrained}, all of the previously mentioned methods are consistent provided that the parameter~$\varepsilon$ and the timestep satisfy~$h\ll \varepsilon$ in the context of methods in~$\R^d$, and satisfy~$\varepsilon\ll h$ in the context of methods on the manifold~$\MM$. When applied in the regime where~$\varepsilon$ and~$h$ share the same order of magnitude, the accuracy of these methods quickly deteriorates, and they may face stability issues.

In past decades, different solutions were proposed for treating the loss of accuracy in the intermediate regime~$h\sim \varepsilon$ in the context of multiscale problems with the help of uniformly accurate (UA) methods.
These methods are capable of solving dynamics indexed by a possibly stiff parameter~$\varepsilon$ with an accuracy and a cost both independent of~$\varepsilon$. A uniformly accurate method is automatically asymptotic-preserving (AP), that is, it converges in the two regimes~$\varepsilon\rightarrow 0$ and~$\varepsilon\sim 1$, but the converse is not true in general.
We refer the reader to the review~\cite{Jin12aps} and references therein for examples of AP integrators for solving multiscale problems.
We mention in particular the paper~\cite{Plechac09imm}, which proposes a penalized Hamiltonian dynamics, and AP discretizations for solving it, and the paper~\cite{Brehier20oap} that gives an AP scheme for the approximation of a class of multiscale SDEs.
In~\cite{Cohen12cao,Vilmart14wso,Laurent20mif}, trigonometric and multirevolution integrators are considered for solving highly oscillatory SDEs (see also the deterministic works~\cite{Melendo97ana,Calvo04aco,Calvo07oem,Chartier14mrc}).
We mention the recent papers~\cite{Chartier15uan,Chartier20anc,Chartier20uan,Almuslimani21uas} (see also the references therein) that introduce uniformly accurate methods for solving a variety of multiscale problems.
There is a rich literature on AP and UA methods, but, to the best of our knowledge, the problem we study here and the techniques we consider are new.
We propose in this paper a new consistent integrator with uniform accuracy and uniform cost for solving penalized Langevin dynamics, that is, a method for solving~\eqref{equation:Langevin_epsilon} whose accuracy and cost do not depend on the parameter~$\varepsilon$.




The article is organized as follows.
Section~\ref{section:main_results} is devoted to the presentation of the new integrator and to the main convergence results.
In Section~\ref{section:proofs}, we build a weak asymptotic expansion of the solution of~\eqref{equation:Langevin_epsilon} that is uniform in~$\varepsilon$, and we use it for proving the uniform accuracy of our integrator.
We compare in Section~\ref{section:numerical_experiments} the new integrator with the explicit Euler scheme in~$\R^d$~\eqref{equation:Euler_EEE} and the constrained Euler scheme~\eqref{equation:Euler_Explicit_constrained} on~$\MM$ in numerical experiments on a torus and on the orthogonal group to confirm its order of convergence in the weak context and for sampling the invariant measure.
Finally, we present some possible extensions and future work in Section~\ref{section:future_work}.

\section{Uniformly accurate integrator for penalized Langevin dynamics}
\label{section:main_results}

In this section, we present the new uniformly accurate integrator and the main convergence results of this paper. The proofs are postponed to Section~\ref{section:proofs}.
Let us first lay down a few notations and assumptions. We assume in the rest of the article that~$\MM=\{x\in\R^d, \zeta(x)=0\}$ is a compact and smooth manifold of codimension~$q\geq 1$ embedded in~$\R^d$, where the constraints are given by the smooth map~$\zeta \colon \R^d\rightarrow\R^q$. We write~$g=\nabla \zeta\colon\R^d\rightarrow\R^{d\times q}$ and we assume that the Gram matrix~$G(x)=g^T(x) g(x)\in\R^{q\times q}$ is invertible for all~$x$ in~$\MM$. With these notations, the projection~$\Pi_\MM$ on the tangent bundle is given by~$\Pi_\MM(x)= I_d-G^{-1}(x) g(x) g^T(x)$.
The test functions typically belong to a subspace of~$\CC^\infty_P(\R^d,\R)$, the vector space of~$\CC^\infty$ functions~$\phi(x)$ such that all partial derivatives up to any order have a polynomial growth of the form
$$\big|\phi^{(k)}(x)\big|\leq C(1+\abs{x}^K),$$
where the constants~$C$ and~$K$ are independent of~$x\in \R^d$ (but can depend on~$k$), and where we denote by~$\abs{x}=(x^Tx)^{1/2}$ the Euclidean norm in~$\R^d$.
Similarly, we denote by~$\CC^p_P(\R^d,\R)$ the space of~$\CC^p$ functions whose partial derivatives up to order~$p$ have polynomial growth.
Letting~$\varphi\colon \R^d\rightarrow \R^{d\times k}$, we use the following notations for differentials: for all vectors~$x$,~$a^1$, \dots,~$a^m\in\R^d$, we denote
$$\varphi^{(m)}(x)(a^1,\dots,a^m)
=\sum_{i_1,\dots,i_m=1}^d \frac{\partial^m\varphi}{\partial x_{i_1}\dots\partial x_{i_m}}(x) \ts a^1_{i_1}\dots a^m_{i_m}.$$
For the sake of clarity, we will often drop the coefficient~$x$, and if~$m=1$, we also use the notation~$\varphi'$ for the Jacobian matrix of~$\varphi$.
Moreover, for~$(e_i)$ the canonical basis of~$\R^d$, we write
$$
\Delta \varphi(x)=\sum_{i=1}^d \varphi''(x)(e_i,e_i)\quad \text{and} \quad (\Div \varphi(x))_j=\sum_{i=1}^d \varphi_{i j}'(x)(e_i).
$$
In the rest of the paper, we make the following assumption in the spirit of the regularity assumptions made in~\cite[Appx.\ts C]{Ciccotti08pod} and~\cite[Chap.\ts 2]{Milstein04snf}.
\begin{ass}
\label{assumption:regularity_ass}
The map~$f$ is bounded, is Lipschitz and lies in~$\CC^3_P$. The maps~$g$ and~$g'$ are bounded in~$\R^d$, and there exists~$c>0$ such that, for~$x$,~$y\in \R^d$,
\begin{equation}
\label{equation:assumption_G_modified}
\abs{G_y(x)}\geq c(1+\abs{y})^{-1},\quad \text{where} \quad G_y(x)=\int_0^1 g^T(x+\tau y)d\tau g(x).
\end{equation}
In addition, there exists a smooth change of coordinate
\begin{align*}
\psi\colon&\R^d\rightarrow \R^d\\
& x \mapsto \begin{pmatrix}
\varphi(x)\\\zeta(x)
\end{pmatrix}
\end{align*}
where~$\varphi\colon\R^d\rightarrow\R^{d-q}$ satisfies~$\varphi'g=0$.
The map~$\psi$ lies in~$\CC^5_P$ and is invertible,~$\psi'$ and~$\psi''$ are Lipschitz and there exist two constants~$c$,~$C>0$ such that~$c\leq\abs{\psi'(x)}\leq C$.
\end{ass}

\begin{remark}
Assumption~\ref{assumption:regularity_ass} is almost the same as the one given in~\cite[Appx.\ts C]{Ciccotti08pod} to prove the strong convergence of the dynamics~\eqref{equation:modified_Langevin_epsilon} to the constrained dynamics~\eqref{equation:projected_Langevin}. The difference lies in the additional estimate~\eqref{equation:assumption_G_modified} that replaces the weaker assumption that the Gram matrix~$G(x)=G_0(x)$ is invertible on~$\MM$.
We use the estimate~\eqref{equation:assumption_G_modified} for obtaining a uniform expansion of the Lagrange multipliers in the new method and for proving that the new method evolves in a neighborhood of the manifold (see Lemma~\ref{lemma:uniform_expansion_lambda}).
Note that the existence of the change of coordinate~$\psi$ is always valid in a neighborhood of the smooth manifold~$\MM$. The same goes for the estimate~\eqref{equation:assumption_G_modified} for~$x$ in a neighborhood of~$\MM$ and~$y$ in a ball centered on zero.
Assumption~\ref{assumption:regularity_ass} is valid in particular if~$\MM$ is a vector subspace of~$\R^d$, but it is quite restrictive. It would be interesting to extend the results of this paper under simpler regularity assumptions made only on the manifold, as numerical experiments hint that the results presented in this paper still stand without global assumptions. This is matter for future work.
\end{remark}

Under Assumption~\ref{assumption:regularity_ass}, the problems~\eqref{equation:projected_Langevin} and~\eqref{equation:Langevin_epsilon} are well posed, and we obtain the strong convergence of the penalized dynamics~\eqref{equation:Langevin_epsilon} to the constrained dynamics~\eqref{equation:projected_Langevin}.
\begin{theorem}
\label{theorem:CV_penalized_constrained}
Under Assumption~\ref{assumption:regularity_ass}, the solution~$X^\varepsilon$ of the penalized dynamics~\eqref{equation:Langevin_epsilon} converges strongly to~$X^0$, the solution of the constrained dynamics~\eqref{equation:projected_Langevin}; that is, for all~$t\leq T$, there exists a constant~$C>0$ such that, for all~$\varepsilon>0$,
$$\sup_{t\leq T}\E\Big[\abs{X^\varepsilon(t)-X^0(t)}^2\Big]\leq C\varepsilon.$$
Moreover,~$\zeta(X^\varepsilon(t))$ satisfies
$$
\sup_{t\leq T}\E\Big[\abs{\zeta(X^\varepsilon(t))}^2\Big]\leq C\varepsilon.
$$
\end{theorem}
This result was first introduced in~\cite[Appx.\ts C]{Ciccotti08pod} for slightly different penalized dynamics. The proof is almost identical, but we present it in Appx.\ts~\ref{section:proof_CV_penalized_constrained} for the sake of completeness. In the deterministic setting with~$\sigma=0$, the convergence to the manifold is of order~$1$ in~$\varepsilon$ instead of order~$1/2$.
We introduce the following additional assumption.
\begin{ass}
\label{assumption:regularity_ass_stronger}
There exists~$c>0$ such that, for~$x$,~$y\in \R^d$,
$$
\abs{G_y(x)}\geq c,\quad \text{where} \quad G_y(x)=\int_0^1 g^T(x+\tau y)d\tau g(x).
$$
\end{ass}
Assumption~\eqref{assumption:regularity_ass_stronger} is a stronger version of the inequality~\eqref{equation:assumption_G_modified} and is in the spirit of the concept of admissible Lagrange multipliers~\cite{Lelievre19hmc}. It is always satisfied for~$x$ in a neighborhood of the manifold~$\MM$ and~$y$ in a ball centered on zero if we assume that the Gram matrix~$G(x)=G_0(x)$ is invertible on~$\MM$. 
We emphasize that we do not need this assumption for proving the uniform accuracy property of the new method, but we use it for obtaining uniform estimates in the regime~$\varepsilon\rightarrow 0$ and on the numerical implementation of the uniformly accurate method.

We introduce the new integrator for approximating the penalized dynamics~\eqref{equation:Langevin_epsilon} with cost and accuracy independent of the parameter~$\varepsilon$, and a cost comparable to that of the constrained Euler scheme~\eqref{equation:Euler_Explicit_constrained} in terms of the number of evaluations of the functions~$f$,~$\zeta$,~$g$, and~$g'$.

\begin{algorithm}[H]
\renewcommand{\thealgorithm}{\methodUA}
\caption{(Uniform discretization of penalized overdamped Langevin dynamics)}
\begin{algorithmic}
\STATE~$X_0^\varepsilon=X_0$ $\in \MM$
\FOR{$n\geq 0$}
\STATE
\vskip-4ex
\begin{align}
X_{n+1}^\varepsilon&=X_n^\varepsilon
+\sqrt{h}\sigma \xi_n
+hf(X_n^\varepsilon)
+\frac{(1-e^{-h/\varepsilon})^2}{2}(g'(g G^{-1} \zeta) G^{-1} \zeta)(X_n^\varepsilon) \nonumber\\
\label{equation:UA_integrator}
&+\frac{\sigma^2 \varepsilon}{8}(1-e^{-2h/\varepsilon})\nabla\ln(\det(G))(X_n^\varepsilon)
+g(X_n^\varepsilon) \lambda^\varepsilon_{n+1},\\
\zeta(X_{n+1}^\varepsilon)&=e^{-h/\varepsilon}\zeta(X_n^\varepsilon)
+\sigma\sqrt{\frac{\varepsilon}{2}(1-e^{-2h/\varepsilon})}g^T(X_n^\varepsilon)\xi_n \nonumber\\
&+\varepsilon(1-e^{-h/\varepsilon})\big(g^T f+\frac{\sigma^2}{4} g^T\nabla\ln(\det(G)) +\frac{\sigma^2}{2}\Div(g)\big)(X_n^\varepsilon) \nonumber\\
&+\sigma^2\Big( \varepsilon(1-e^{-h/\varepsilon})-\sqrt{\frac{\varepsilon h}{2}(1-e^{-2h/\varepsilon})} \Big) \nonumber\\
&\times \Big(\sum_{i=1}^d (g'(gG^{-1}g^Te_i))^T gG^{-1}g^Te_i - \sum_{i=1}^d (g'(e_i))^T gG^{-1}g^Te_i \Big)(X_n^\varepsilon). \nonumber
\end{align}
\ENDFOR
\end{algorithmic}
\end{algorithm}

The new method works in a way similar to the constrained Euler integrator~\eqref{equation:Euler_Explicit_constrained}. Knowing the approximation~$X_n^\varepsilon$ of~$X^\varepsilon(nh)$, we project a modified Euler step on a modified manifold defined by the constraint given in~\eqref{equation:UA_integrator} (in place of~$\zeta(X_{n+1}^\varepsilon)=0$ for the constrained Euler integrator~\eqref{equation:Euler_Explicit_constrained}). We project the modified step in the direction~$g(X_n^\varepsilon)$ with the help of a Lagrange multiplier~$\lambda^\varepsilon_{n+1}$. For the implementation of the method, one can use, for instance, a fixed point iteration or a Newton method at each step to find the solution~$(X_{n+1}^\varepsilon,\lambda^\varepsilon_{n+1})$ of the implicit system of equations~\eqref{equation:UA_integrator}.

In order for the discretization~\eqref{equation:UA_integrator} to be well-defined, we use bounded random variables.
The~$\xi_n$ are independent and bounded discrete random vectors that have the same moments as standard Gaussian random vectors up to order four, in the spirit of~\cite[Chap.\ts 2]{Milstein04snf}, that is, for instance, that their components satisfy
\begin{equation}
\label{equation:discrete_rv}
\P(\xi_i=0)=\frac{2}{3} \quad \text{and} \quad \P(\xi_i=\pm\sqrt{3})=\frac{1}{6}, \quad i=1,\dots,d.
\end{equation}
Note that using truncated Gaussian random variables would also work.

\begin{remark}
The new integrator is related to the popular idea of backward error analysis and modified equations for SDEs (see, for instance,~\cite{Zygalakis11ote,Abdulle12hwo,Debussche12wbe,Kopec15wbea,Kopec15wbeb}).
The idea is to define a projection method (see~\cite[Sect.\ts IV.4]{Hairer06gni}) with a modified constraint in place of~$\zeta(X_n)=0$.
Instead of evaluating the stiff term~$\frac{h}{\varepsilon} g G^{-1} \zeta$ as in the Euler scheme~\eqref{equation:Euler_EEE}, we project a modified step of the explicit Euler scheme in~$\R^d$ on a manifold that is close to~$\MM$ when~$\varepsilon\ll 1$ and whose constraint is given by a truncation of a uniform expansion of~$\zeta(X^\varepsilon)$.
When~$\varepsilon\rightarrow \infty$, the expression of the constraint~$\zeta(X_{n+1}^\varepsilon)$ in~\eqref{equation:UA_integrator} tends to a truncated Taylor expansion in~$h$ around~$X_n^\varepsilon$, while for~$\varepsilon\rightarrow 0$,~$\zeta(X_{n+1}^\varepsilon)$ tends to zero, which enforces that the integrator lies on~$\MM$. These intuitions will be made rigorous in Section~\ref{section:proofs}.
\end{remark}

\begin{remark}
In the context of a manifold~$\MM$ of codimension~$q=1$, the Gram matrix~$G(x)$ and~$g(x)^Te_i=g_i(x)$ are real numbers, so that~$\nabla\ln(\det(G))=2G^{-1} g'(g)$ and
$$
\sum_{i=1}^d (g'(gG^{-1}g^Te_i))^T gG^{-1}g^Te_i
=G^{-1}(g'(g))^T g
=\sum_{i=1}^d (g'(e_i))^T gG^{-1}g^Te_i.
$$
The discretization~\eqref{equation:UA_integrator} thus reduces to
\begin{align}
X_{n+1}^\varepsilon&=X_n^\varepsilon
+\sqrt{h}\sigma \xi_n
+hf(X_n^\varepsilon)
+\frac{(1-e^{-h/\varepsilon})^2}{2}(\zeta^2 G^{-2} g'(g))(X_n^\varepsilon) \nonumber\\
\label{equation:UA_integrator_q_1}
&+\frac{\sigma^2 \varepsilon}{4}(1-e^{-2h/\varepsilon})(G^{-1} g'(g))(X_n^\varepsilon)
+g(X_n^\varepsilon) \lambda^\varepsilon_{n+1},\\
\zeta(X_{n+1}^\varepsilon)&=e^{-h/\varepsilon}\zeta(X_n^\varepsilon)
+\sigma\sqrt{\frac{\varepsilon}{2}(1-e^{-2h/\varepsilon})}g^T(X_n^\varepsilon)\xi_n \nonumber\\
&+\varepsilon(1-e^{-h/\varepsilon})(g^T f +\frac{\sigma^2}{2} G^{-1} g^T g'(g) +\frac{\sigma^2}{2}\Div(g))(X_n^\varepsilon). \nonumber
\end{align}
\end{remark}

We present in the rest of the section the uniform accuracy property of the discretization~\eqref{equation:UA_integrator}, and we show that the integrator converges to the constrained Euler scheme~\eqref{equation:Euler_Explicit_constrained} when~$\varepsilon\rightarrow 0$.
The different convergence results are summarized by the following commutative diagram, where~$T=Nh$ is fixed. Note that, as we present a convergence result in~$h$ that is uniform in~$\varepsilon$, the two arrows for the convergence in~$h$ rely on the same Theorem~\ref{theorem:uniform_consistency_UA_scheme}.
$$
\begin{tikzcd}[row sep=large,column sep = large]
    \text{integrator } X^\varepsilon_N \text{ in } \R^d
    \arrow[d,"h\rightarrow 0"',"(Thm.\ts~\ref{theorem:uniform_consistency_UA_scheme})"]
    \arrow[r,"\varepsilon\rightarrow 0","(Thm.\ts~\ref{theorem:convergence_integrator_epsilon})"']
    & \text{integrator } X^0_N \text{ on } \MM
    \arrow[d,"h\rightarrow 0"',"(Thm.\ts~\ref{theorem:uniform_consistency_UA_scheme})"]
    \\
    \text{solution }X^\varepsilon(T) \text{ of~\eqref{equation:Langevin_epsilon}}
    \arrow[r,"\varepsilon\rightarrow 0","(Thm.\ts~\ref{theorem:CV_penalized_constrained})"']
    & \text{solution }X^0(T) \text{ of~\eqref{equation:projected_Langevin}}
\end{tikzcd}
$$

We now state the main result of this work, that is, the uniform accuracy of the discretization~\eqref{equation:UA_integrator} for approximating the solution of the penalized Langevin dynamics~\eqref{equation:Langevin_epsilon}.
\begin{theorem}
\label{theorem:uniform_consistency_UA_scheme}
Under Assumption~\ref{assumption:regularity_ass}, the integrator~$(X_n^\varepsilon)$ given by~\eqref{equation:UA_integrator} is a consistent uniformly accurate approximation of the solution~$X^\varepsilon(t)$ of the penalized Langevin dynamics~\eqref{equation:Langevin_epsilon}; that is, for a given test function~$\phi\in \CC^5_P$, there exist~$h_0>0$,~$C>0$ such that for all~$\varepsilon>0$,~$h\leq h_0$, the following estimate holds:
\begin{equation}
\label{equation:uniform_global_consistency}
\abs{\E[\phi(X_n^\varepsilon)]-\E[\phi(X^\varepsilon(nh))]}\leq C\sqrt{h},\quad n=0,1,\dots,N,\quad Nh=T.
\end{equation}
\end{theorem}

\begin{remark}
\label{remark:discussion_order_UA_method}
Note that Theorem~\ref{theorem:uniform_consistency_UA_scheme} states the uniform consistency, but not the uniform weak order one, as one could expect.
The discretization~\eqref{equation:UA_integrator} has weak order one if~$\varepsilon=\varepsilon_0$ is fixed or in the limit~$\varepsilon\rightarrow 0$, but an order reduction occurs in the intermediate regime~$0<\varepsilon<\varepsilon_0$, and the integrator only has weak order~$1/2$ with respect to~$h$ in general.
For the sake of simplicity, we leave the creation of uniformly accurate integrators of higher weak order for future works.
\end{remark}

We present the proof of Theorem~\ref{theorem:uniform_consistency_UA_scheme} in Section~\ref{section:proofs}.
It relies on a weak expansion in~$h$ of the solution of~\eqref{equation:Langevin_epsilon} that is uniform in~$\varepsilon$.
One could directly use this uniform expansion as an explicit numerical integrator for solving~\eqref{equation:Langevin_epsilon}. It would also yield a uniformly accurate scheme and would not require one to solve a fixed point problem.
However, in the limit~$\varepsilon\rightarrow 0$, this integrator would almost surely not stay on the manifold.
The crucial geometric property that the integrator lies on the manifold when~$\varepsilon\rightarrow 0$ is satisfied for the new method, as stated in the following result.
\begin{theorem}
\label{theorem:convergence_integrator_epsilon}
Under Assumption~\ref{assumption:regularity_ass}, the integrator~$(X_n^\varepsilon)$ in~\eqref{equation:UA_integrator} converges to the Euler scheme on the manifold~\eqref{equation:Euler_Explicit_constrained} when~$\varepsilon\rightarrow 0$; that is, for~$h$ and~$N$ fixed such that~$T=Nh$, there exists a constant~$C_h>0$ that depends on~$h$ but not on~$\varepsilon$ such that
$$
\abs{\zeta(X^\varepsilon_n)}\leq C_h\sqrt{\varepsilon},\quad n=0,1,\dots,N.
$$
In addition, if Assumption~\ref{assumption:regularity_ass_stronger} is satisfied, then, for~$h_0$ small enough, for~$h\leq h_0$ fixed, there exists a constant~$C_h>0$ such that
\begin{equation}
\label{equation:weak_CV_epsilon_0_strong}
\abs{X_n^\varepsilon-X_n^0}\leq C_h \sqrt{\varepsilon},\quad n=0,1,\dots,N,\quad Nh=T.
\end{equation}
\end{theorem}



\begin{remark}
In the deterministic context (i.e.\ts when~$\sigma=0$), the uniform accuracy of the discretization~\eqref{equation:UA_integrator} still holds, and the speed of convergence to the manifold~$\MM$ of both the exact solution and the integrator are in~$\OO(\varepsilon)$.
To the best of our knowledge, the integrator given by~\eqref{equation:UA_integrator} is the first integrator with the uniform accuracy property for solving the singular perturbation problem~\eqref{equation:Langevin_epsilon} with~$\sigma=0$.
However, similar expansions that are uniform with respect to~$\varepsilon$ are presented in~\cite[Chaps.\ts VI-VII]{Hairer10sod} and references therein.
\end{remark}

\begin{remark}
Another widely used scheme on the manifold is the Euler scheme with implicit projection direction,
\begin{equation}
\label{equation:Euler_Implicit_constrained}
X^0_{n+1}=X^0_n+hf(X^0_n) +\sqrt{h}\sigma \xi_n + g(X^0_{n+1})\lambda^0_{n+1}, \quad \zeta(X^0_{n+1})=0,
\end{equation}
where the Lagrange multiplier~$\lambda^0_{n+1}$ is determined by the constraint~$\zeta(X^0_{n+1})=0$.
The uniformly accurate discretization given in~\eqref{equation:UA_integrator} can be modified so that it converges to the integrator~\eqref{equation:Euler_Implicit_constrained} when~$\varepsilon\rightarrow 0$.
It suffices to replace the first line of~\eqref{equation:UA_integrator}
\begin{align*}
X_{n+1}^\varepsilon&=X_n^\varepsilon
+\sqrt{h}\sigma \xi_n
+hf(X_n^\varepsilon)
-\frac{(1-e^{-h/\varepsilon})^2}{2}(g'(g G^{-1} \zeta) G^{-1} \zeta)(X_n^\varepsilon) \nonumber\\
&+\frac{\sigma^2}{4}\Big(\sqrt{2\varepsilon h(1-e^{-2h/\varepsilon})}-\frac{\varepsilon}{2}(1-e^{-2h/\varepsilon})\Big)\nabla\ln(\det(G))(X_n^\varepsilon)
+g(X_{n+1}^\varepsilon) \lambda^\varepsilon_{n+1},
\end{align*}
and to keep the same expansion for the constraint~$\zeta(X_{n+1}^\varepsilon)$. The methodology for the uniform expansion of the integrator that we present in Section~\ref{section:expansion_num_sol} extends to this context, so that the convergence results persist.
Similarly, one could change the direction of projection~$g(X_n^\varepsilon)$ into~$g(Y_n^\varepsilon)$, where~$Y_n^\varepsilon$ is any consistent one-step approximation of~$X_n^\varepsilon$, in the spirit of the class of projected Runge-Kutta methods presented in~\cite{Laurent21ocf}.
Finding a class of uniformly accurate discretizations that converge to a more general class of Runge-Kutta methods on the manifold~$\MM$ is matter for future works.
\end{remark}

The uniform discretization~\eqref{equation:UA_integrator} is implicit and requires one to solve a fixed point problem at each step with, for instance, a fixed point iteration or a Newton method.
The following result, in the spirit of~\cite[Chap.\ts VII]{Hairer10sod} for deterministic DAEs and~\cite[Lemma 3.3]{Laurent21ocf} for the constrained dynamics~\eqref{equation:projected_Langevin}, confirms that the associated implicit system is not stiff, that is, that its complexity does not depend on the stiff parameter~$\varepsilon$.
\begin{theorem}
\label{theorem:uniform_fixed_point_problem}
Under Assumption~\ref{assumption:regularity_ass}, each step of the integrator~$(X_n^\varepsilon)$ given by~\eqref{equation:UA_integrator} can be rewritten as a solution of a fixed point problem of the form
$$
X_{n+1}^\varepsilon=F_h^\varepsilon(X_{n+1}^\varepsilon),
$$
where~$F_h^\varepsilon\colon\R^d\rightarrow\R^d$ depends on~$X_n^\varepsilon$,~$\xi_n$,~$h$, and~$\varepsilon$.
Moreover, if Assumption~\ref{assumption:regularity_ass_stronger} is satisfied, then there exists~$h_0>0$ independent of~$\varepsilon$ such that for all~$h\leq h_0$,~$F_h^\varepsilon$ is a uniform contraction, that is, there exists a positive constant~$L< 1$ independent of~$h$ and~$\varepsilon$ such that, for all~$y_1$,~$y_2\in\R^d$,
$$\abs{F_h^\varepsilon(y_2)-F_h^\varepsilon(y_1)}\leq L\abs{y_2-y_1}.$$
\end{theorem}

\section{Weak convergence analysis}
\label{section:proofs}

In this section, we present the uniform weak expansion and the stability properties of the solution~$X^\varepsilon(t)$ to the penalized dynamic~\eqref{equation:Langevin_epsilon} and of the uniform integrator~$(X_n^\varepsilon)$ given in~\eqref{equation:UA_integrator}. We then use these results to prove Theorem~\ref{theorem:uniform_consistency_UA_scheme}, Theorem~\ref{theorem:convergence_integrator_epsilon} and Theorem~\ref{theorem:uniform_fixed_point_problem}.
Let us begin the analysis with a few technical lemmas and notations.

\begin{lemma}
\label{lemma:estimate_uniform_zeta}
Let~$(X_n^\varepsilon)$ be given by~\eqref{equation:UA_integrator}, then, under Assumption~\ref{assumption:regularity_ass}, there exists a constant~$C_0>0$ independent of~$X_0 \in \MM$,~$\varepsilon$, and~$h$ such that, for all~$n\geq 0$,
\begin{equation}
\label{equation:convergence_zeta_numeric_lemma}
(1-e^{-h/\varepsilon})\abs{\zeta(X^\varepsilon_n)}\leq C_0\sqrt{h}.
\end{equation}
\end{lemma}

\begin{proof}
As~$f$,~$\xi_n$,~$g$, and~$g'$ are bounded, we obtain from the definition of the integrator given by~\eqref{equation:UA_integrator} that
$$
|\zeta(X^\varepsilon_{n+1})|\leq e^{-h/\varepsilon} |\zeta(X^\varepsilon_n)|
+C \sqrt{\varepsilon(1-e^{-2h/\varepsilon})},
$$
where we used that the function~$(1-e^{-x})/x$ is bounded for~$x>0$.
Thus, as~$\zeta(X_0)=0$,
\begin{align*}
|\zeta(X^\varepsilon_n)|
&\leq e^{-nh/\varepsilon}|\zeta(X_0)|
+C\sqrt{\varepsilon(1-e^{-2h/\varepsilon})} \sum_{k=0}^{n-1} e^{-kh/\varepsilon}\\
&\leq C \frac{\sqrt{\varepsilon(1-e^{-2h/\varepsilon})}}{1-e^{-h/\varepsilon}}
\leq C_0  \frac{\sqrt{h}}{1-e^{-h/\varepsilon}},
\end{align*}
where the constant~$C_0$ does not depend on~$X_0$,~$\varepsilon$,~$n$, and~$h$. This yields the estimate~\eqref{equation:convergence_zeta_numeric_lemma}.
\end{proof}

The estimate~\eqref{equation:convergence_zeta_numeric_lemma} is a direct consequence of our choice of using bounded random variables in~\eqref{equation:UA_integrator}.
We shall use this estimate extensively in the rest of this section.
Thus, we denote by~$\MM^\varepsilon_h$ the set of vectors~$x\in \R^d$ that satisfy the estimate
\begin{equation}
\label{equation:convergence_zeta_numeric}
(1-e^{-h/\varepsilon})\abs{\zeta(x)}\leq C_0\sqrt{h},
\end{equation}
where~$C_0$ is the constant given in Lemma~\ref{lemma:estimate_uniform_zeta}. The set~$\MM^\varepsilon_h$ is a closed subset of~$\R^d$ that contains~$\MM$. The numerical scheme given by~\eqref{equation:UA_integrator} takes values in~$\MM^\varepsilon_h$. We mention that the convergence results are still valid if the initial condition~$X_0$ of~\eqref{equation:UA_integrator} is chosen in~$\MM^\varepsilon_h$ instead of~$\MM$.

As we aim at writing uniform expansions, we introduce the convenient notation~$R^\varepsilon_h(x)$ for any remainder that satisfies at least~$\abs{\E[R^\varepsilon_h(x)]}\leq C h^{3/2}$, where~$C$ is independent of~$\varepsilon$,~$h$, and~$x$.

The following result serves as a technical tool for simplifying the calculations in the uniform expansions of Subsection~\ref{section:expansion_exact_sol} and Subsection~\ref{section:expansion_num_sol}. It is proved with elementary computations.
\begin{lemma}
\label{lemma:rewriting_fixman_correction}
The Fixman correction can be rewritten in the following way:
$$
\frac{\sigma^2}{4}\nabla\ln(\det(G))
=\frac{\sigma^2}{2}\sum_{i=1}^d g'(e_i) G^{-1}g^T e_i
=\frac{\sigma^2}{2}\sum_{i=1}^d g'(g G^{-1}g^T e_i) G^{-1}g^T e_i,
$$
where~$(e_i)$ is the canonical basis of~$\R^d$.
\end{lemma}

\begin{proof}[Proof of Lemma~\ref{lemma:rewriting_fixman_correction}]
Using that~$G=g^T g$ is a symmetric matrix, that~$g=\nabla \zeta$ is a gradient (which implies~$x^Tg'(y)=y^Tg'(x)$)  and the standard properties of the trace operator~$\Trace$, we deduce that
\begin{align*}
\partial_j (\ln(\det(G)))
&=\Trace(G^{-1}\partial_j G)
=2\Trace(G^{-1}(g'(e_j))^T g)
=2\Trace(g'(e_j) G^{-1} g^T)\\
&=2\sum_{i=1}^d e_i^T g'(e_j) G^{-1} g^T e_i
=2\sum_{i=1}^d e_j^T g'(e_i) G^{-1} g^T e_i,
\end{align*}
that is,~$\nabla\ln(\det(G))
=2\sum_{i=1}^d g'(e_i) G^{-1}g^T e_i$. For the second equality, we have
\begin{align*}
\sum_{i=1}^d e_j^T g'(g G^{-1}g^T e_i) G^{-1}g^T e_i
&=\sum_{i=1}^d e_i^T g G^{-1} g^T g'(e_j) G^{-1}g^T e_i\\
&=\Trace(g G^{-1} g^T g'(e_j) G^{-1}g^T)
=\Trace(g'(e_j) G^{-1}g^T)\\
&=\sum_{i=1}^d e_i^T g'(e_j) G^{-1}g^T e_i
=\sum_{i=1}^d e_j^T g'(e_i) G^{-1} g^T e_i.
\end{align*}
Hence we get the result.
\end{proof}

\subsection{Uniform expansion of the exact solution}
\label{section:expansion_exact_sol}

We consider the exact solution~$X^\varepsilon(t)$ of the penalized Langevin dynamics~\eqref{equation:Langevin_epsilon}, with the initial condition~$X_0=x$, that we assume is deterministic for simplicity. Then,~$X^\varepsilon$ satisfies the following expansion in~$h$ that is uniform with
respect to~$\varepsilon$.
\begin{proposition}
\label{proposition:weak_uniform_expansion_penalized_dynamic}
Under Assumption~\ref{assumption:regularity_ass}, there exists~$h_0>0$ such that for all~$h\leq h_0$, if~$X^\varepsilon$ is the solution of the penalized Langevin dynamics~\eqref{equation:Langevin_epsilon} starting at~$x\in\MM^\varepsilon_h$, then, for all~$\phi\in \CC^3_P$, the following estimate holds:
\begin{equation}
\label{equation:weak_expansion_penalized_dynamic}
\abs{\E[\phi(X^\varepsilon(h))]-\E[\phi(x+\sqrt{h}A^\varepsilon_h(x)+h B^\varepsilon_h(x))]}\leq C(1+\abs{x}^K) h^{3/2},
\end{equation}
where~$C$ is independent of~$h$ and~$\varepsilon$ and where the functions~$A^\varepsilon_h$ and~$B^\varepsilon_h$ are given by
\begin{align*}
A^\varepsilon_h&=\sigma \xi
+\frac{e^{-h/\varepsilon}-1}{\sqrt{h}} g G^{-1} \zeta
+\sigma\Big(\sqrt{\frac{\varepsilon}{2h}(1-e^{-2h/\varepsilon})}-1\Big)g G^{-1}g^T \xi,\\
B^\varepsilon_h&=f
+\Big(\frac{\varepsilon}{h}(1-e^{-h/\varepsilon})-1\Big)gG^{-1}g^T f
+\frac{\sigma^2}{2}\Big(\frac{\varepsilon}{h}(1-e^{-h/\varepsilon})-1\Big) gG^{-1} \Div(g)\\
&+\frac{\sigma^2 \varepsilon}{8h}(1-e^{-2h/\varepsilon})\nabla\ln(\det(G))
+\frac{\sigma^2 \varepsilon}{8h}(1-e^{-h/\varepsilon})^2 g G^{-1} g^T\nabla\ln(\det(G))\\
&+\frac{1}{2h}(e^{-h/\varepsilon}-1)^2 \Big(g'(g G^{-1} \zeta) G^{-1} \zeta
-g G^{-1} g^T g'(g G^{-1} \zeta) G^{-1} \zeta
\\&-gG^{-1}(g'(g G^{-1} \zeta))^T g G^{-1} \zeta\Big)
+\sigma^2\Big(1+\frac{\varepsilon}{h}(e^{-h/\varepsilon}-1)\Big)\sum_{i=1}^d  g G^{-1}(g'(e_i))^T g G^{-1}g^T e_i\\
&+\frac{\sigma^2}{4} \Big(\frac{\varepsilon}{h}(e^{-2h/\varepsilon}-4e^{-h/\varepsilon}+3)-2\Big) \sum_{i=1}^d gG^{-1}(g'(g G^{-1}g^T e_i))^T g G^{-1}g^T e_i,
\end{align*}
with~$\xi$ a discrete bounded random vector that satisfies~\eqref{equation:discrete_rv}, and where the functions~$A^\varepsilon_h$ and~$B^\varepsilon_h$ are bounded uniformly in~$\varepsilon$ and~$h$ on~$\MM^\varepsilon_h$.
\end{proposition}

Note that for~$q=1$, we obtain the first step of the uniformly accurate discretization~\eqref{equation:UA_integrator_q_1} by gathering all the terms of the form~$gM$ with~$M\in\R^q$ of the weak approximation given in Proposition~\ref{proposition:weak_uniform_expansion_penalized_dynamic} in a Lagrange multiplier~$g \lambda^\varepsilon_1\in \R^q$ and by adding the truncated expansion of~$\zeta(X^\varepsilon(h))$ as a constraint.

The proof of Proposition~\ref{proposition:weak_uniform_expansion_penalized_dynamic} relies on the change of coordinate~$\psi$ given in Assumption~\ref{assumption:regularity_ass}.
Instead of discretizing directly the penalized dynamics~\eqref{equation:Langevin_epsilon}, we first apply the change of coordinate~$\psi$ and we derive an expansion in time of~$X^\varepsilon(t)$ that is uniform in the parameter~$\varepsilon$.
The following result is used for proving Proposition~\ref{proposition:weak_uniform_expansion_penalized_dynamic}.
\begin{lemma}
\label{lemma:strong_expansion_penalized_dynamic}
With the same notations and assumptions as in Proposition~\ref{proposition:weak_uniform_expansion_penalized_dynamic}, the following estimates hold for all~$h\leq h_0$ and all $x\in\MM^\varepsilon_h$:
\begin{align}
\label{equation:strong_expansion_0}
\E[|X^\varepsilon(h)-x|^2]^{1/2}&\leq C\sqrt{h},\\
\label{equation:strong_expansion_1/2}
\E[|X^\varepsilon(h)-(x+\sqrt{h}\widehat{A}^\varepsilon_h(x))|^2]^{1/2}&\leq Ch,\\
\label{equation:weak_auxiliary_expansion_penalized_dynamic_psi}
\abs{\E[\psi(X^\varepsilon(h))]-\E[\psi(x+\sqrt{h}\widehat{A}^\varepsilon_h(x)+h B^\varepsilon_h(x))]}&\leq Ch^{3/2},
\end{align}
where~$C$ is independent of~$h$ and~$\varepsilon$,~$B^\varepsilon_h$ is defined in Proposition~\ref{proposition:weak_uniform_expansion_penalized_dynamic}, and~$\widehat{A}^\varepsilon_h$ is given by
\begin{align*}
\widehat{A}^\varepsilon_h&=\sigma \frac{W(h)}{\sqrt{h}}+\frac{e^{-h/\varepsilon}-1}{\sqrt{h}} g G^{-1} \zeta
+\frac{\sigma}{\sqrt{h}}g G^{-1}g^T\int_0^h (e^{(s-h)/\varepsilon}-1) dW(s).
\end{align*}
\end{lemma}

\begin{proof}[Proof of Proposition~\ref{proposition:weak_uniform_expansion_penalized_dynamic}]
The uniform bounds on~$A^\varepsilon_h(X^\varepsilon(t))$ and~$B^\varepsilon_h(X^\varepsilon(t))$ are obtained straightforwardly by using Assumption~\ref{assumption:regularity_ass} and the fact that~$x\in\MM^\varepsilon_h$.
We prove the local weak order one of the approximation~$Y^\varepsilon(h)=x+\sqrt{h}\widehat{A}^\varepsilon_h(x)+h B^\varepsilon_h(x)$ given in Lemma~\ref{lemma:strong_expansion_penalized_dynamic}.
Let~$\phi\in \CC^3_P$ and~$\widetilde{\phi}=\phi\circ \psi^{-1}$; then a Taylor expansion around~$x$ yields
\begin{align*}
\Big|\E[\phi(X^\varepsilon(h))-\phi(Y^\varepsilon(h))]\Big|
&= \Big|\E[\widetilde{\phi}\circ \psi(X^\varepsilon(h))]-\E[\widetilde{\phi}\circ \psi(Y^\varepsilon(h))]\Big|\\
&\leq 
\Big|\E[\widetilde{\phi}'(\psi(x))(\psi(X^\varepsilon(h))-\psi(Y^\varepsilon(h)))]\Big|\\
&+\frac{1}{2}\Big|\E[\widetilde{\phi}''(\psi(x))(\psi(X^\varepsilon(h))-\psi(x),\psi(X^\varepsilon(h))-\psi(x))\\
&-\widetilde{\phi}''(\psi(x))(\psi(Y^\varepsilon(h))-\psi(x),\psi(Y^\varepsilon(h))-\psi(x))]\Big|\\
&+C(1+\abs{x}^K)h^{3/2}\\
&\leq 
\Big|\widetilde{\phi}'(\psi(x))\Big| \Big|\E[\psi(X^\varepsilon(h))-\psi(Y^\varepsilon(h))]\Big|\\
&+\frac{1}{2}\Big|\widetilde{\phi}''(\psi(x))\Big|
\E[|\psi(X^\varepsilon(h))+\psi(Y^\varepsilon(h))-2\psi(x)|^2]^{1/2}\\
&\cdot \E[|\psi(X^\varepsilon(h))-\psi(Y^\varepsilon(h))|^2]^{1/2}
+C(1+\abs{x}^K)h^{3/2}
\end{align*}
where we used~\eqref{equation:strong_expansion_0}, Assumption~\ref{assumption:regularity_ass}, and the bilinearity of~$\widetilde{\phi}''(\psi(x))$.
With Lemma~\ref{lemma:strong_expansion_penalized_dynamic} and the regularity properties of~$\widetilde{\phi}$ and~$\psi$, we get
\begin{equation}
\label{equation:weak_auxiliary_expansion_penalized_dynamic_proof}
\abs{\E[\phi(X^\varepsilon(h))]-\E[\phi(x+\sqrt{h}\widehat{A}^\varepsilon_h(x)+h B^\varepsilon_h(x))]}\leq  C(1+\abs{x}^K)h^{3/2}.
\end{equation}

In the spirit of~\cite[Chap.\ts 2]{Milstein04snf}, we replace the random variable~$\widehat{A}^\varepsilon_h(x)$ by the random variable~$A^\varepsilon_h(x)$ that share the same expectation and covariance matrix.
Indeed, a calculation gives
\begin{align*}
\Cov(\widehat{A}^\varepsilon_{h,i}(x),\widehat{A}^\varepsilon_{h,j}(x))
&=\sigma^2 \delta_{ij}
+\frac{2\sigma^2}{h} (g G^{-1}g^T)_{ij} \int_0^h (e^{(s-h)/\varepsilon}-1)ds\\
&+\frac{\sigma^2}{h}\sum_{k=1}^d (g G^{-1}g^T)_{ik}(g G^{-1}g^T)_{jk} \int_0^h (e^{(s-h)/\varepsilon}-1)^2 ds\\
&=\sigma^2 \delta_{ij}+\sigma^2 (g G^{-1}g^T)_{ij} \Big(\frac{\varepsilon}{2h}(1-e^{-2h/\varepsilon})-1\Big),
\end{align*}
where we used that $g^T g=G$ and the Itô isometry.
On the other hand, a similar calculation yields
\begin{align*}
\Cov(A^\varepsilon_{h,i}(x),A^\varepsilon_{h,j}(x))
&=\sigma^2 \delta_{ij}
+2\sigma^2 \Big(\sqrt{\frac{\varepsilon}{2h}(1-e^{-2h/\varepsilon})}-1\Big) (g G^{-1}g^T)_{ij}\\
&+\sigma^2\Big(\sqrt{\frac{\varepsilon}{2h}(1-e^{-2h/\varepsilon})}-1\Big)^2 \sum_{k=1}^d (g G^{-1}g^T)_{ik} (g G^{-1}g^T)_{jk}\\
&=\sigma^2 \delta_{ij}+\sigma^2 (g G^{-1}g^T)_{ij} \Big(\frac{\varepsilon}{2h}(1-e^{-2h/\varepsilon})-1\Big).
\end{align*}
Replacing $\widehat{A}^\varepsilon_h(x)$ by $A^\varepsilon_h(x)$ in the weak expansion~\eqref{equation:weak_auxiliary_expansion_penalized_dynamic_proof} gives the estimate~\eqref{equation:weak_expansion_penalized_dynamic}.
\end{proof}

The main ingredient of the proof of Lemma~\ref{lemma:strong_expansion_penalized_dynamic} is the decomposition of the terms of the expansion in a part that stays on the tangent space and a part of the form~$gM$ with~$M\in\R^q$ that is orthogonal to the tangent space.
\begin{proof}[Proof of Lemma~\ref{lemma:strong_expansion_penalized_dynamic}]
As~$\psi^{-1}$ is Lipschitz, we have
\begin{align*}
\E[|X^\varepsilon(h)-x|^2]^{1/2}
&\leq C\E[|\psi(X^\varepsilon(h))-\psi(x)|^2]^{1/2}\\
&\leq C\E[|\varphi(X^\varepsilon(h))-\varphi(x)|^2]^{1/2}
+C\E[|\zeta(X^\varepsilon(h))-\zeta(x)|^2]^{1/2}.
\end{align*}
On the one hand, applying the Itô formula to~$\varphi(X^\varepsilon)$ yields
\begin{align}
\label{equation:varphi_Ito_formula_integral}
\varphi(X^\varepsilon(h))
&=\varphi(x)
+\sigma \int_0^h \varphi'(X^\varepsilon(s)) dW(s)\\
&+\int_0^h [\varphi'f+\frac{\sigma^2}{4}\varphi'\nabla\ln(\det(G))+\frac{\sigma^2}{2}\Delta \varphi](X^\varepsilon(s)) ds, \nonumber
\end{align}
where the term in~$\varepsilon$ vanishes as~$\varphi'g=0$ (see Assumption~\ref{assumption:regularity_ass}).
Assumption~\ref{assumption:regularity_ass} allows us to write the uniform strong expansion
$$\E[|\varphi(X^\varepsilon(h))-\varphi(x)|^2]^{1/2}\leq C\sqrt{h}.$$
On the other hand, for~$\zeta(X^\varepsilon)$, we have
$$
d\zeta(X^\varepsilon)=\sigma g^T(X^\varepsilon) dW+\Big[g^T f+\frac{\sigma^2}{4}g^T\nabla\ln(\det(G))+\frac{\sigma^2}{2}\Div(g)-\frac{1}{\varepsilon}\zeta\Big](X^\varepsilon) dt.
$$
With the variation of constants formula, it rewrites into
\begin{align}
\label{equation:zeta_Ito_formula_integral}
\zeta(X^\varepsilon(h))&=
e^{-h/\varepsilon}\zeta(x)
+\sigma \int_0^h e^{(s-h)/\varepsilon} g^T (X^\varepsilon(s)) dW(s)\\
&+\int_0^h e^{(s-h)/\varepsilon}\Big[g^T f+\frac{\sigma^2}{4}g^T\nabla\ln(\det(G))+\frac{\sigma^2}{2}\Div(g)\Big](X^\varepsilon(s)) ds. \nonumber
\end{align}
As the integrands in~\eqref{equation:zeta_Ito_formula_integral} are bounded (using Assumption~\ref{assumption:regularity_ass}), we get 
$$\E[|\zeta(X^\varepsilon(h))-\zeta(x)|^2]^{1/2}
\leq C((e^{-h/\varepsilon}-1)^2 \abs{\zeta(x)}^2+h)^{1/2}
\leq C\sqrt{h},$$
where we used that~$x\in\MM^\varepsilon_h$.
We thus get the desired estimate~\eqref{equation:strong_expansion_0}.
The estimate~\eqref{equation:strong_expansion_1/2} is obtained with the same arguments by keeping track of the terms of size~$\OO(\sqrt{h})$ in the expansions.

We now prove the weak estimate~\eqref{equation:weak_auxiliary_expansion_penalized_dynamic_psi}.
We denote for simplicity~$Y^\varepsilon(h)=x+\sqrt{h}\widehat{A}^\varepsilon_h(x)+h B^\varepsilon_h(x)$.
Let us first look at the approximation of~$\varphi(X^\varepsilon(h))$.
On the one hand, applying the Itô formula to~$\varphi(X^\varepsilon(t))$ gives
\begin{align*}
\varphi(X^\varepsilon(h))
&=\varphi(x)
+h\varphi'f(x)+h\frac{\sigma^2}{4}\varphi'\nabla\ln(\det(G))(x)+h\frac{\sigma^2}{2}\Delta \varphi(x)
+R^\varepsilon_h(x),
\end{align*}
where we used~\eqref{equation:strong_expansion_0} and we put in~$R^\varepsilon_h(x)$ all the terms that are zero in average.
On the other hand, an expansion in~$h$ of~$\varphi(Y^\varepsilon(h))$ yields
\begin{align*}
\varphi(Y^\varepsilon(h))
&=\varphi 
+\sqrt{h}\varphi'\widehat{A}^\varepsilon_h
+h\Big[\varphi'B^\varepsilon_h+\frac{1}{2}\varphi''(\widehat{A}^\varepsilon_h,\widehat{A}^\varepsilon_h)\Big]
+R^\varepsilon_h\\
&=\varphi
+h\varphi'f
+\frac{\sigma^2 \varepsilon}{8}(1-e^{-2h/\varepsilon})\varphi'\nabla\ln(\det(G))
+\frac{1}{2}(e^{-h/\varepsilon}-1)^2 \varphi' g'(g G^{-1} \zeta) G^{-1} \zeta\\
&+\frac{\sigma^2}{2}\varphi'' (W(h),W(h))
+\frac{1}{2}(e^{-h/\varepsilon}-1)^2\varphi''(g G^{-1} \zeta,g G^{-1} \zeta) \\
&+\frac{\sigma^2}{2}\int_0^h \int_0^h (e^{(s-h)/\varepsilon}-1) (e^{(u-h)/\varepsilon}-1)\varphi''(g G^{-1}g^T dW(s),g G^{-1}g^T dW(u)) \\
&+\sigma^2\int_0^h (e^{(s-h)/\varepsilon}-1) \varphi''(W(h),g G^{-1}g^T dW(s))
+R^\varepsilon_h,
\end{align*}
where we use that~$\varphi'g=0$, we omit the dependence in~$x$ for conciseness, and we put in~$R^\varepsilon_h(x)$ all the terms that are zero in average.
We now replace the random terms by their expectation,
\begin{align*}
\varphi(Y^\varepsilon(h))
&=\varphi 
+h\varphi'f 
+\frac{\sigma^2 \varepsilon}{8}(1-e^{-2h/\varepsilon})\varphi'\nabla\ln(\det(G)) 
+h\frac{\sigma^2}{2}\Delta\varphi \\
&+\frac{1}{2}(e^{-h/\varepsilon}-1)^2 (\varphi'(g'(g G^{-1} \zeta) G^{-1} \zeta)+\varphi''(g G^{-1} \zeta,g G^{-1} \zeta)) \\
&+\frac{\sigma^2}{2} \int_0^h (e^{(s-h)/\varepsilon}-1)^2 ds \sum_{i=1}^d \varphi''(g G^{-1}g^T e_i,g G^{-1}g^T e_i) \\
&+\sigma^2 \int_0^h (e^{(s-h)/\varepsilon}-1) ds \sum_{i=1}^d \varphi''(e_i,g G^{-1}g^T e_i) 
+R^\varepsilon_h.
\end{align*}
Letting~$M\in\R^q$, we differentiate the equality~$\varphi'(gM)=0$. We obtain that for any~$M\in\R^q$ and~$v\in\R^d$ we have~$\varphi'(g'(v)M)+\varphi''(gM,v)=0$. We deduce that
$$\varphi'(g'(g G^{-1} \zeta) G^{-1} \zeta)+\varphi''(g G^{-1} \zeta,g G^{-1} \zeta)=0.$$
Applying Lemma~\ref{lemma:rewriting_fixman_correction}, we get by a direct calculation that
$$\varphi(Y^\varepsilon(h))=\varphi(x)
+h\varphi'f(x)+h\frac{\sigma^2}{4}\varphi'\nabla\ln(\det(G))(x)+h\frac{\sigma^2}{2}\Delta \varphi(x)
+R^\varepsilon_h(x),$$
which gives the desired estimate
\begin{equation}
\label{equation:weak_auxiliary_expansion_penalized_dynamic_varphi}
\abs{\E[\varphi(X^\varepsilon(h))]-\E[\varphi(x+\sqrt{h}\widehat{A}^\varepsilon_h(x)+h B^\varepsilon_h(x))]}\leq Ch^{3/2}.
\end{equation}

For the one-step approximation of~$\zeta(X^\varepsilon(h))$, the Itô formula and the variation of constants formula yield
$$
\zeta(X^\varepsilon(h))
=
e^{-h/\varepsilon}\zeta(x)
+\varepsilon(1-e^{-h/\varepsilon})(g^T f+\frac{\sigma^2}{4}g^T \nabla\ln(\det(G))+\frac{\sigma^2}{2}\Div(g))(x)
+R^\varepsilon_h(x).
$$
For~$\zeta(Y^\varepsilon(h))$, using~$g^T g=G$, we get with the same arguments as for~$\varphi(Y^\varepsilon(h))$ that
\begin{align*}
\zeta(Y^\varepsilon(h))
&=\zeta 
+\sqrt{h}g^T \widehat{A}^\varepsilon_h 
+h\Big[g^T B^\varepsilon_h +\frac{1}{2}(g'(\widehat{A}^\varepsilon_h))^T \widehat{A}^\varepsilon_h \Big]
+R^\varepsilon_h \\
&=
e^{-h/\varepsilon}\zeta
+\varepsilon(1-e^{-h/\varepsilon})g^T f
+\frac{\sigma^2}{2}\Big(\varepsilon(1-e^{-h/\varepsilon})-h\Big) \Div(g)+\frac{\sigma^2}{2} (g' (W(h)))^T W(h)\\
&+\frac{\sigma^2 \varepsilon}{4}(1-e^{-h/\varepsilon})^2 g^T\nabla\ln(\det(G))
+\sigma^2\Big(h+\varepsilon(e^{-h/\varepsilon}-1)\Big)\sum_{i=1}^d (g'(e_i))^T g G^{-1}g^T e_i\\
&+\frac{\sigma^2}{4} \Big(\varepsilon(e^{-2h/\varepsilon}-4e^{-h/\varepsilon}+3)-2h\Big) \sum_{i=1}^d(g'(g G^{-1}g^T e_i))^T g G^{-1}g^T e_i\\
&+\frac{\sigma^2}{2} \int_0^h \int_0^h (e^{(s-h)/\varepsilon}-1)(e^{(u-h)/\varepsilon}-1) (g'(g G^{-1}g^T dW(u)))^T g G^{-1}g^T dW(s) \\
&+\sigma^2 \int_0^h \int_0^h (e^{(s-h)/\varepsilon}-1) (g'(W(h)))^T g G^{-1}g^T dW(s)  
+R^\varepsilon_h.
\end{align*}
We now replace the stochastic integrals by their expectations (putting the remainders in~$R^\varepsilon_h$), and we use Lemma~\ref{lemma:rewriting_fixman_correction} to simplify the expansion. It yields
\begin{align*}
\zeta(Y^\varepsilon(h))
&=e^{-h/\varepsilon}\zeta 
+\varepsilon(1-e^{-h/\varepsilon})(g^T f+\frac{\sigma^2}{4}g^T \nabla\ln(\det(G))+\frac{\sigma^2}{2}\Div(g)) 
+R^\varepsilon_h,
\end{align*}
which implies
\begin{equation}
\label{equation:weak_auxiliary_expansion_penalized_dynamic_zeta}
\abs{\E[\zeta(X^\varepsilon(h))]-\E[\zeta(x+\sqrt{h}\widehat{A}^\varepsilon_h(x)+h B^\varepsilon_h(x))]}\leq Ch^{3/2}.
\end{equation}
Combining the inequalities~\eqref{equation:weak_auxiliary_expansion_penalized_dynamic_varphi} and~\eqref{equation:weak_auxiliary_expansion_penalized_dynamic_zeta} gives the desired weak estimate~\eqref{equation:weak_auxiliary_expansion_penalized_dynamic_psi}.
\end{proof}

To end this subsection, we recall the growth properties of~$X^\varepsilon$ that will be of use in Subsection~\ref{section:proofs_cv_theorems}.
For this particular result, we add the dependency in the initial condition~$x$ of the exact solution of~\eqref{equation:Langevin_epsilon} with the notation~$X^\varepsilon(t,x)$.
\begin{lemma}
\label{lemma:growth_properties_exact_sol}
Under Assumption~\ref{assumption:regularity_ass}, for~$\phi\in \CC^5_P$ and~$t\leq T$ fixed, the function~$\widetilde{\phi}^\varepsilon(x)=\E[\phi( X^\varepsilon(t,x))]$ lies in~$\CC^3_P$ with constants independent of~$t$ and~$\varepsilon$.
\end{lemma}

\begin{proof}
The standard theory (see, for instance, the textbook~\cite{Milstein04snf}) gives~$\widetilde{\phi}^\varepsilon\in \CC^3$.
With the regularity assumptions on~$\psi$ and its derivatives, it is sufficient to prove that
$
\E[\phi(Y^\varepsilon(t,y))]
$
is in~$\CC^3_P$, where~$Y^\varepsilon(t,y)=\psi(X^\varepsilon(t,\psi^{-1}(x)))$ (replacing~$\phi$ by~$\phi\circ \psi^{-1}$).
We recall from the proof of Lemma~\ref{lemma:strong_expansion_penalized_dynamic} that~$\varphi(X^\varepsilon(t,x))$ satisfies the integral formulation~\eqref{equation:varphi_Ito_formula_integral} and that~$\zeta(X^\varepsilon(t,x))$ satisfies~\eqref{equation:zeta_Ito_formula_integral}.
Putting together these equations, we deduce that~$Y^\varepsilon(t,y)$ satisfies an equation of the form
$$
Y^\varepsilon(t,y)=A^\varepsilon(t) y+\int_0^t A^\varepsilon(t-s) F(Y^\varepsilon(s,y))ds+\int_0^t A^\varepsilon(t-s) G(Y^\varepsilon(s,y))dW(s),$$
where~$F$ and~$G$ are in~$\CC^3_P$ and do not depend on~$\varepsilon$, and
$$A^\varepsilon(t)=\begin{pmatrix}
I_{d-q} &0\\
0&e^{-t/\varepsilon}I_q
\end{pmatrix}.$$
As~$\abs{A^\varepsilon(t)}\leq C$, the process~$Y^\varepsilon(t,y)$ satisfies~$\E[\abs{Y^\varepsilon(t,y)}^{2p}]\leq C(1+\abs{y}^K)$, with~$C$ and~$K$ independent of~$\varepsilon$.
Thus we have
$$\E[\phi(Y^\varepsilon(t,y))]\leq C(1+\E[\abs{Y^\varepsilon(t,y)}^K])\leq C(1+\E[\abs{y}^K]).$$
For the derivatives, we recall from~\cite{Gikhman07sde} that~$Z^\varepsilon(t,y)=\partial_y Y^\varepsilon(t,y)$ satisfies the equation
\begin{align*}
Z^\varepsilon(t,y)&=A^\varepsilon(t) I_d+\int_0^t A^\varepsilon(t-s) F'(Y^\varepsilon(s,y))Z^\varepsilon(t,y)ds\\
&+\int_0^t A^\varepsilon(t-s) G'(Y^\varepsilon(s,y))Z^\varepsilon(t,y)dW(s).
\end{align*}
Applying the same arguments as for~$Y^\varepsilon(t,y)$ yields that~$Z^\varepsilon(t,y)$ has bounded moments of all order and that~$\E[\phi(Z^\varepsilon(t,y))]\leq C(1+\E[\abs{y}^K])$.
The same methodology extends to~$\partial_y^2 Y^\varepsilon(t,y)$ and~$\partial_y^3 Y^\varepsilon(t,y)$.
\end{proof}

\subsection{Uniform expansion and bounded moments of the numerical solution}
\label{section:expansion_num_sol}

In this subsection, we show that the integrator given by~\eqref{equation:UA_integrator} has bounded moments of all order, that it lies on the manifold~$\MM$ in the limit~$\varepsilon\rightarrow 0$, and that it satisfies the same local weak uniform expansion as the exact solution of~\eqref{equation:Langevin_epsilon} (see Proposition~\ref{proposition:weak_uniform_expansion_penalized_dynamic}).

First, the integrator given by~\eqref{equation:UA_integrator} satisfies the following bounded moments property.
\begin{proposition}
\label{proposition:bounded_moments}
Under Assumption~\ref{assumption:regularity_ass},~$(X_n^\varepsilon)$ has bounded moments of any order along time; i.e., for all timestep~$h\leq h_0$ small enough such that~$Nh=T$ is fixed, for all integer~$m\geq 0$,
$$
\sup_{n\leq N} \E[\abs{X_n^\varepsilon}^{2m}] \leq C_m,
$$
where the constant~$C>0$ is independent of~$\varepsilon$ and~$h$.
\end{proposition}

The integrator~$(X_n^\varepsilon)$ given in~\eqref{equation:UA_integrator} also satisfies a uniform local expansion that is similar to its continuous counterpart presented in Proposition~\ref{proposition:weak_uniform_expansion_penalized_dynamic}.
\begin{proposition}
\label{proposition:uniform_expansion_integrator}
Under Assumption~\ref{assumption:regularity_ass}, there exists~$h_0>0$ such that for all~$h\leq h_0$, if~$(X^\varepsilon_n)$ is the numerical discretization given by~\eqref{equation:UA_integrator} beginning at~$x\in \MM^\varepsilon_h$ (assumed deterministic for simplicity), then, for all test functions~$\phi\in \CC^3_P$, the following estimate holds:
\begin{equation}
\label{equation:weak_expansion_numerical method}
\abs{\E[\phi(X^\varepsilon_1)]-\E[\phi(x+\sqrt{h}A^\varepsilon_h(x)+h B^\varepsilon_h(x))]}\leq C(1+\abs{x}^K) h^{3/2},
\end{equation}
where~$C$ is independent of~$h$ and~$\varepsilon$, and~$A^\varepsilon_h$ and~$B^\varepsilon_h$ are the functions given in Proposition~\ref{proposition:weak_uniform_expansion_penalized_dynamic}.
\end{proposition}

To prove Proposition~\ref{proposition:bounded_moments} and Proposition~\ref{proposition:uniform_expansion_integrator}, we rely on the following lemma, whose proof is postponed to the end of this subsection.
We emphasize that an inequality of the form $\abs{\zeta(x)}\leq C$ does not imply in general that $x$ stays close to $\MM$. That is why we rely in Lemma \ref{lemma:uniform_expansion_lambda} on an estimate of the Lagrange multipliers (using the inequality \eqref{equation:assumption_G_modified}).
This estimate ensures that the method evolves in a neighborhood of the manifold.
\begin{lemma}
\label{lemma:uniform_expansion_lambda}
Under Assumption~\ref{assumption:regularity_ass} and if~$x\in\MM^\varepsilon_h$, there exists~$h_0>0$ such that, for all timestep~$h\leq h_0$, the one-step approximation~$X_1^\varepsilon$ and the Lagrange multiplier~$\lambda^\varepsilon_1$ in the discretization~\eqref{equation:UA_integrator} satisfy
$$
\abs{X_1^\varepsilon-x}\leq C\sqrt{h},
\qquad
\lambda^\varepsilon_1=\sqrt{h}G^{-1}(x)\lambda_{1,(1/2)}^{\varepsilon}+hG^{-1}(x)\lambda_{1,(1)}^{\varepsilon}+R^\varepsilon_h(x),
$$
where~$|\lambda_{1,(1/2)}^{\varepsilon}|\leq C$,~$|\lambda_{1,(1)}^{\varepsilon}|\leq C$, and~$|R^\varepsilon_h(x)|\leq Ch^{3/2}$ with~$C$ independent of~$\varepsilon$,~$h$ and~$x$.
\end{lemma}

For proving Proposition~\ref{proposition:bounded_moments}, we apply the change of variable~$\psi$ and we adapt the standard methodology presented in~\cite[Lemma 1.1.6 \& Lemma 2.2.2]{Milstein04snf}.
\begin{proof}[Proof of Proposition~\ref{proposition:bounded_moments}]
We derive from~\eqref{equation:UA_integrator} that
\begin{align*}
\abs{\E[\zeta(X_{n+1}^\varepsilon)-e^{-h/\varepsilon}\zeta(X_n^\varepsilon)|X_n^\varepsilon]}&\leq C h,\\
\abs{\zeta(X_{n+1}^\varepsilon)-e^{-h/\varepsilon}\zeta(X_n^\varepsilon)}&\leq C \sqrt{h}.
\end{align*}
We prove that~$\zeta(X_n)$ has bounded moments by induction on~$n$. The binomial formula yields
\begin{align*}
\E[\abs{\zeta(X_{n+1}^\varepsilon)}^{2m}]
&=\E[\abs{e^{-h/\varepsilon}\zeta(X^\varepsilon_n)+\zeta(X^\varepsilon_{n+1})-e^{-h/\varepsilon}\zeta(X^\varepsilon_n)}^{2m}]\\
&\leq e^{-2mh/\varepsilon} \E[\abs{\zeta(X_n^\varepsilon)}^{2m}]\\
&+C\E[\abs{\zeta(X_n^\varepsilon)}^{2m-1}\abs{\E[\zeta(X_{n+1}^\varepsilon)-e^{-h/\varepsilon}\zeta(X_n^\varepsilon)|X_n^\varepsilon]}]\\
&+C\sum_{k=2}^{2m} \E[\abs{\zeta(X_n^\varepsilon)}^{2m-k} \abs{\zeta(X_{n+1}^\varepsilon)-e^{-h/\varepsilon}\zeta(X_n^\varepsilon)}^k]\\
&\leq \E[\abs{\zeta(X_n^\varepsilon)}^{2m}]+C(1+\E[\abs{\zeta(X_n^\varepsilon)}^{2m}])h.
\end{align*}
Following~\cite[Lemma 1.1.6]{Milstein04snf}, as~$X_0\in\MM$ is bounded, it implies that~$\E[\abs{\zeta(X_n)}^{2m}]$ is bounded uniformly in~$n=0,\dots,N$ and~$\varepsilon$.

Using Lemma~\ref{lemma:uniform_expansion_lambda} and the equality~$\varphi'g=0$, a direct calculation gives
\begin{align}
\label{equation:bounded_moments_sufficient_condition_1}
\abs{\E[\varphi(X_{n+1}^\varepsilon)-\varphi(X_n^\varepsilon)|X_n^\varepsilon]}&\leq C h,\\
\label{equation:bounded_moments_sufficient_condition_2}
\abs{\varphi(X_{n+1}^\varepsilon)-\varphi(X_n^\varepsilon)}&\leq C \sqrt{h}.
\end{align}
Following the same methodology as for~$\E[\abs{\zeta(X_n)}^{2m}]$, the estimates~\eqref{equation:bounded_moments_sufficient_condition_1}-\eqref{equation:bounded_moments_sufficient_condition_2} imply that the quantity~$\E[\abs{\varphi(X_n)}^{2m}]$ is bounded uniformly in~$n=0,\dots,N$ and~$\varepsilon$.
Then, as~$\psi^{-1}$ is Lipschitz, we have
$$
\E[\abs{X_n^\varepsilon}^{2m}]\leq C(1+\E[\abs{\psi(X_n^\varepsilon)}^{2m}])\leq C(1+\E[\abs{\varphi(X_n^\varepsilon)}^{2m}]+\E[\abs{\zeta(X_n^\varepsilon)}^{2m}])\leq C.
$$
Hence we get the result.
\end{proof}

We obtain the uniform expansion of the numerical solution by writing explicitly a uniform expansion of the Lagrange multiplier~$\lambda^\varepsilon_1$, in the spirit of~\cite[Lemma 3.25]{Lelievre10fec}.
\begin{proof}[Proof of Proposition~\ref{proposition:uniform_expansion_integrator}]
Using Lemma~\ref{lemma:uniform_expansion_lambda} and Assumption~\ref{assumption:regularity_ass}, we obtain
$$
X^\varepsilon_1=x+\sqrt{h}\Big[
\sigma \xi
+(gG^{-1})(x)\lambda_{1,(1/2)}^{\varepsilon}
\Big]+R_h^\varepsilon(x),
$$
where the remainder satisfies~$\abs{R_h^\varepsilon(x)}\leq Ch$.
The constraint is then given by
\begin{align*}
\zeta(X^\varepsilon_1)
&=\zeta(x)+\sqrt{h}\Big[
\sigma g^T(x)\xi
+\lambda_{1,(1/2)}^{\varepsilon}
\Big]+R_h^\varepsilon(x).
\end{align*}
On the other hand, we get from the definition of the integrator~\eqref{equation:UA_integrator} that
\begin{align*}
\zeta(X^\varepsilon_1)&=\zeta(x)+\sqrt{h}\Big[
\frac{e^{-h/\varepsilon}-1}{\sqrt{h}}\zeta(x)
+\sigma\sqrt{\frac{\varepsilon}{2h}(1-e^{-2h/\varepsilon})}g^T(x)\xi
\Big]+R_h^\varepsilon(x).
\end{align*}
By identifying the two terms in~$\sqrt{h}$ in the expansions of~$\zeta(X^\varepsilon_1)$, we deduce the value of~$\lambda_{1,(1/2)}^{\varepsilon}$, that is,
$$
\lambda_{1,(1/2)}^{\varepsilon}=
\frac{e^{-h/\varepsilon}-1}{\sqrt{h}} \zeta
+\sigma\Big(\sqrt{\frac{\varepsilon}{2h}(1-e^{-2h/\varepsilon})}-1\Big)g^T \xi.
$$
The expression of~$\lambda_{1,(1/2)}^{\varepsilon}$ and Lemma~\ref{lemma:uniform_expansion_lambda} give
\begin{align}
X^\varepsilon_1&=x+\sqrt{h}A^\varepsilon_h(x)+h\Big[f(x)
+\frac{(1-e^{-h/\varepsilon})^2}{2h}(g'(g G^{-1} \zeta) G^{-1} \zeta)(x)\nonumber\\
\label{equation:expansion_1_X_proof}
&+\frac{\sigma^2 \varepsilon}{8h}(1-e^{-2h/\varepsilon})\nabla\ln(\det(G))(x)+(gG^{-1})(x)\lambda_{1,(1)}^{\varepsilon}
\Big] +R_h^\varepsilon(x),
\end{align}
where~$\abs{R_h^\varepsilon(x)}\leq Ch^{3/2}$.
We then compute the expansion of~$\zeta(X^\varepsilon_1)$, and we compare it with the definition of the integrator~\eqref{equation:UA_integrator} to obtain the expression of~$\lambda_{1,(1)}^{\varepsilon}$.
Inserting this expression in~\eqref{equation:expansion_1_X_proof} gives
$$
X^\varepsilon_1=x+\sqrt{h}A^\varepsilon_h(x)+hB^\varepsilon_h(x)+R_h^\varepsilon(x),
$$
where the remainder satisfies~$\abs{\E[R_h^\varepsilon(x)]}\leq Ch^{3/2}$ and~$\abs{R_h^\varepsilon(x)}\leq Ch$.
A Taylor expansion of~$\phi(X^\varepsilon_1)$ around~$x+\sqrt{h}A^\varepsilon_h(x)+h B^\varepsilon_h(x)$ yields the estimate~\eqref{equation:weak_expansion_numerical method}.
\end{proof}

The proof of Lemma~\ref{lemma:uniform_expansion_lambda} mainly relies on the estimates~\eqref{equation:assumption_G_modified} and~\eqref{equation:convergence_zeta_numeric}.
We refer the reader to~\cite[Lemma 3.3]{Laurent21ocf} and~\cite[Chap.\ts VII]{Hairer10sod} for similar proofs where explicit expressions of Lagrange multipliers are derived.
\begin{proof}[Proof of Lemma~\ref{lemma:uniform_expansion_lambda}]
Let~$x\in\MM^\varepsilon_h$. For brevity, we rewrite the discretization~\eqref{equation:UA_integrator} as
\begin{align*}
X_1^\varepsilon&=x+\sqrt{h}Y_{(1/2)}^\varepsilon(x)+hY_{(1)}^\varepsilon(x)+g(x)\lambda^\varepsilon_1,\\
\zeta(X_1^\varepsilon)&=\zeta(x)+\sqrt{h}\zeta_{(1/2)}^\varepsilon(x)+h\zeta_{(1)}^\varepsilon(x),
\end{align*}
where the functions~$Y_{(1/2)}^\varepsilon$,~$Y_{(1)}^\varepsilon$,~$\zeta_{(1/2)}^\varepsilon$, and~$\zeta_{(1)}^\varepsilon$ are given by
\begin{align*}
Y_{(1/2)}^\varepsilon &=\sigma\xi,\\
Y_{(1)}^\varepsilon &=f 
+\frac{(1-e^{-h/\varepsilon})^2}{2h}(g'(g G^{-1} \zeta) G^{-1} \zeta) 
+\frac{\sigma^2 \varepsilon}{8h}(1-e^{-2h/\varepsilon})\nabla\ln(\det(G)) ,\\
\zeta_{(1/2)}^\varepsilon &=\frac{e^{-h/\varepsilon}-1}{\sqrt{h}}\zeta 
+\sigma\sqrt{\frac{\varepsilon}{2h}(1-e^{-2h/\varepsilon})}g^T \xi,\\
\zeta_{(1)}^\varepsilon &=
\frac{\varepsilon}{h}(1-e^{-h/\varepsilon})\big(g^T f+\frac{\sigma^2}{4} g^T\nabla\ln(\det(G)) +\frac{\sigma^2}{2}\Div(g)\big) \\
&+\sigma^2\Big(\frac{\varepsilon}{h}(1-e^{-h/\varepsilon})-\sqrt{\frac{\varepsilon}{2h}(1-e^{-2h/\varepsilon})} \Big) \sum_{i=1}^d ((g'(gG^{-1}g^Te_i))^T gG^{-1}g^Te_i) \\
&-\sigma^2\Big(\frac{\varepsilon}{h}(1-e^{-h/\varepsilon})-\sqrt{\frac{\varepsilon}{2h}(1-e^{-2h/\varepsilon})} \Big) \sum_{i=1}^d ((g'(e_i))^T gG^{-1}g^Te_i) .
\end{align*}
Using Assumption~\ref{assumption:regularity_ass} and the estimate~\eqref{equation:convergence_zeta_numeric}, the following uniform estimates hold:
$$
\abs{Y_{(i)}^\varepsilon(x)}\leq C,\quad \abs{\zeta_{(i)}^\varepsilon(x)}\leq C,\quad i\in\{1/2,1\}.
$$
The fundamental theorem of calculus yields
$$
\zeta(X_1^\varepsilon)-\zeta(x)=\int_0^1 g^T(x+\tau(X_1^\varepsilon-x))d\tau(X_1^\varepsilon-x)=\sqrt{h}\zeta_{(1/2)}^\varepsilon(x)+h\zeta_{(1)}^\varepsilon(x).
$$
Substituting~$X_1^\varepsilon-x$ then gives
$$
\int_0^1 g^T(x+\tau(X_1^\varepsilon-x))d\tau (\sqrt{h}Y_{(1/2)}^\varepsilon(x)+hY_{(1)}^\varepsilon(x)+g(x)\lambda^\varepsilon_1)
=\sqrt{h}\zeta_{(1/2)}^\varepsilon(x)+h\zeta_{(1)}^\varepsilon(x).
$$
Using Assumption~\ref{assumption:regularity_ass}, we get the following explicit expression of~$\lambda^\varepsilon_1$:
\begin{align}
\label{equation:expression_exact_lambda}
\lambda^\varepsilon_1&=\sqrt{h}G_{X_1^\varepsilon-x}^{-1}(x)
\Big(\zeta_{(1/2)}^\varepsilon(x)-\int_0^1 g^T(x+\tau(X_1^\varepsilon-x))d\tau Y_{(1/2)}^\varepsilon(x)\Big)\\
&+hG_{X_1^\varepsilon-x}^{-1}(x)
\Big(\zeta_{(1)}^\varepsilon(x)-\int_0^1 g^T(x+\tau(X_1^\varepsilon-x))d\tau Y_{(1)}^\varepsilon(x)\Big). \nonumber
\end{align}
Then, the growth assumption~\eqref{equation:assumption_G_modified} on~$G_y^{-1}(x)$ allows us to write
$$\abs{\lambda^\varepsilon_1}\leq C\sqrt{h} (1+\abs{X_1^\varepsilon-x}) \quad \text{and} \quad \abs{X_1^\varepsilon-x}\leq C\sqrt{h}(1+\abs{X_1^\varepsilon-x}).$$
Hence, for~$h\leq h_0$ small enough, we deduce that~$\abs{X_1^\varepsilon-x}\leq C\sqrt{h}$ and
\begin{equation}
\label{equation:expansion_1/2_lambda}
\lambda^\varepsilon_1= \sqrt{h}G^{-1}(x)
(\zeta_{(1/2)}^\varepsilon-g^T Y_{(1/2)}^\varepsilon)(x)+R^\varepsilon_h,
\end{equation}
where~$\abs{R^\varepsilon_h}\leq Ch$.
For the term of size~$\OO(h)$, we first deduce from~\eqref{equation:expansion_1/2_lambda} that 
$$\abs{X_1^\varepsilon-x-\sqrt{h}(Y_{(1/2)}^\varepsilon+gG^{-1}
\zeta_{(1/2)}^\varepsilon-gG^{-1}g^T Y_{(1/2)}^\varepsilon)(x)}\leq Ch.$$
By using this estimate in~\eqref{equation:expression_exact_lambda}, a Taylor expansion yields the desired expansion of~$\lambda^\varepsilon_1$.
\end{proof}

%
%

\subsection{Proofs of the convergence theorems}
\label{section:proofs_cv_theorems}

Now that we have the local uniform expansion of the exact solution and the numerical scheme, as well as the stability property of Proposition~\ref{proposition:bounded_moments}, we are able to prove the main convergence theorems.
\begin{proof}[Proof of Theorem~\ref{theorem:uniform_consistency_UA_scheme}]


We derive the global weak consistency~\eqref{equation:uniform_global_consistency} with techniques similar to the ones presented in~\cite[Chap.\ts 2]{Milstein04snf}.
We denote by~$X^\varepsilon(t,x)$ the solution of the penalized dynamics with initial condition~$x$ and~$X_n^\varepsilon(x)$ the numerical solution with initial condition~$x$.
For~$x\in\MM^\varepsilon_h$, Proposition~\ref{proposition:weak_uniform_expansion_penalized_dynamic} and Proposition~\ref{proposition:uniform_expansion_integrator} yield
$$
\abs{\E[\phi(X^\varepsilon(h,x))-\phi(X_1^\varepsilon(x))|x]}\leq C (1+\abs{x}^K) h^{3/2},
$$
where~$\phi\in \CC^3_P$.
Lemma~\ref{lemma:growth_properties_exact_sol} gives that~$\phi_n(x)=\E[\phi(X^\varepsilon((n-1)h,x))|x]$ is in~$\CC^3_P$.
We rewrite the global error, given by~$E^\varepsilon_h=\abs{\E[\phi(X^\varepsilon(T,X_0))-\phi(X_N^\varepsilon(X_0))]}$, with a telescopic sum,
\begin{align*}
E^\varepsilon_h
&\leq \sum_{n=1}^N \abs{\E[\phi(X^\varepsilon(nh,X^\varepsilon_{N-n}(X_0)))-\phi(X^\varepsilon((n-1)h,X^\varepsilon_{N-n+1}(X_0)))]}\\
&\leq \sum_{n=1}^N \abs{\E[\phi_n(X^\varepsilon(h,X^\varepsilon_{N-n}(X_0)))-\phi_n(X^\varepsilon_1(X^\varepsilon_{N-n}(X_0)))]}\\
&\leq \sum_{n=1}^N \E[\abs{\E[\phi_n(X^\varepsilon(h,x))-\phi_n(X^\varepsilon_1(x))|x=X^\varepsilon_{N-n}(X_0)]}]\\
&\leq \sum_{n=1}^N C (1+\E[\abs{X^\varepsilon_{N-n}(X_0)}^K]) h^{3/2}
\leq C h^{1/2},
\end{align*}
where we used the bounded moments property of Proposition~\ref{proposition:bounded_moments} and~$X_n^\varepsilon\in \MM^\varepsilon_h$ (Lemma~\ref{lemma:estimate_uniform_zeta}).
\end{proof}

\begin{proof}[Proof of Theorem~\ref{theorem:convergence_integrator_epsilon}]
We obtain straightforwardly from the expression of the integrator~\eqref{equation:UA_integrator} that~$\abs{\zeta(X_{n}^\varepsilon)}\leq C_h\sqrt{\varepsilon}$.
Using this estimate and the notation introduced in the proof of Lemma~\ref{lemma:uniform_expansion_lambda}, we observe that
$$
\abs{Y_{(i)}^\varepsilon(X_n^\varepsilon)-Y_{(i)}^0(X_n^\varepsilon)}\leq C_h\sqrt{\varepsilon},\quad \abs{\zeta_{(i)}^\varepsilon(X_n^\varepsilon)}\leq C_h\sqrt{\varepsilon},\quad i\in\{1/2,1\},
$$
where~$Y_{(1/2)}^0(x)=\sigma\xi_n$ and~$Y_{(1)}^0(x)=f(x)$.
The Lagrange multiplier given by~\eqref{equation:expression_exact_lambda} therefore satisfies~$\abs{\lambda^\varepsilon_{n+1}-\widetilde{\lambda^\varepsilon_{n+1}}}\leq C_h \sqrt{\varepsilon}$, where
$$
\widetilde{\lambda^\varepsilon_{n+1}}=-G_{X_{n+1}^\varepsilon-X_n^\varepsilon}^{-1}(X_n^\varepsilon)\int_0^1 g^T(X_n^\varepsilon+\tau(X_{n+1}^\varepsilon-X_n^\varepsilon))d\tau(\sigma\sqrt{h}\xi_n+hf(X_n^\varepsilon)).
$$
Similarly to~\eqref{equation:expression_exact_lambda}, the Lagrange multiplier~$\lambda^0_{n+1}$ of the constrained Euler integrator~\eqref{equation:Euler_Explicit_constrained} is given by
$$
\lambda^0_{n+1}=-G_{X_{n+1}^0-X_n^0}^{-1}(X_n^0)\int_0^1 g^T(X_n^0+\tau(X_{n+1}^0-X_n^0))d\tau(\sigma\sqrt{h}\xi_n+hf(X_n^0)).
$$
Using that~$G_y^{-1}(x)$ is bounded, a straightforward calculation shows that~$G_y^{-1}(x)$ is Lipschitz in~$x$,~$y\in\R^d$, that is, there exists a constant~$L>0$ such that
$$
\abs{G_{y_1}^{-1}(x_1)-G_{y_2}^{-1}(x_2)}\leq L(\abs{x_1-x_2}+\abs{y_1-y_2}), \quad x_1,x_2,y_1,y_2\in\R^d.
$$
Thus, we get the estimate
$$
\abs{\lambda^\varepsilon_{n+1}-\lambda^0_{n+1}}\leq C_h (\sqrt{\varepsilon} +\abs{X_n^\varepsilon-X_n^0})+C\sqrt{h}\abs{X_{n+1}^\varepsilon-X_{n+1}^0},
$$
and, as~$\lambda^\varepsilon_{n+1}$ and~$\lambda^0_{n+1}$ are bounded uniformly in~$\varepsilon$, we have
$$
\abs{g(X_n^\varepsilon)\lambda^\varepsilon_{n+1}-g(X_n^0)\lambda^0_{n+1}}\leq C_h (\sqrt{\varepsilon} +\abs{X_n^\varepsilon-X_n^0})+C\sqrt{h}\abs{X_{n+1}^\varepsilon-X_{n+1}^0}.
$$
From the definitions of~$X_{n+1}^\varepsilon$ in~\eqref{equation:UA_integrator} and~$X_{n+1}^0$ in~\eqref{equation:Euler_Explicit_constrained}, we deduce that
$$\abs{X_{n+1}^\varepsilon-X_{n+1}^0}\leq C_h (\sqrt{\varepsilon} +\abs{X_n^\varepsilon-X_n^0})+C\sqrt{h}\abs{X_{n+1}^\varepsilon-X_{n+1}^0}.$$
If~$h\leq h_0$ is small enough, we obtain 
$$\abs{X_{n+1}^\varepsilon-X_{n+1}^0}\leq C_h (\sqrt{\varepsilon} +\abs{X_n^\varepsilon-X_n^0}).$$
We deduce the estimate~\eqref{equation:weak_CV_epsilon_0_strong} by iterating this inequality for~$n\leq N=T/h$.
\end{proof}

\begin{proof}[Proof of Theorem~\ref{theorem:uniform_fixed_point_problem}]
We take over the notations and the expression~\eqref{equation:expression_exact_lambda} of the Lagrange multiplier~$\lambda^\varepsilon_1$ that we used in the proof of Lemma~\ref{lemma:uniform_expansion_lambda}. Without loss of generality, we concentrate on the first step of the algorithm with~$X_0=x$.
Replacing~$\lambda^\varepsilon_1$ by the explicit formula~\eqref{equation:expression_exact_lambda} in~\eqref{equation:UA_integrator} yields that~$X_1^\varepsilon$ is a fixed point of the following map:
\begin{align*}
F_h^\varepsilon(y)&=x
+\sqrt{h}Y_{(1/2)}^\varepsilon(x)+hY_{(1)}^\varepsilon(x)
+g(x)G_{y-x}^{-1}(x)
\Big[\sqrt{h}\zeta_{(1/2)}^\varepsilon(x)+h\zeta_{(1)}^\varepsilon(x)\\
&-\int_0^1 g^T(x+\tau(y-x))d\tau (\sqrt{h}Y_{(1/2)}^\varepsilon(x)+hY_{(1)}^\varepsilon(x))\Big].
\end{align*}
For~$y_1$,~$y_2\in\R^d$, Assumption~\ref{assumption:regularity_ass} gives
\begin{align*}
\abs{F_h^\varepsilon(y_2)-F_h^\varepsilon(y_1)}
&\leq C\sqrt{h}\abs{y_2-y_1},
\end{align*}
where we used that~$G_y^{-1}(x)$ is Lipschitz in~$x$,~$y\in\R^d$.
We deduce that~$F_h^\varepsilon$ is a uniform contraction for~$h\leq h_0$ small enough.
\end{proof}

\section{Numerical experiments}
\label{section:numerical_experiments}

In this section, we perform numerical experiments to confirm the theoretical findings, on a torus in~$\R^3$ and on the orthogonal group in high dimension and codimension, in the spirit of the experiments in~\cite{Zappa18mco,Zhang19eso,Laurent21ocf}.

\subsection{Uniform approximation for the invariant measure on a torus}

We consider the example in codimension one of a torus in~$\R^3$.
We apply the new method given by the discretization~\eqref{equation:UA_integrator_q_1} for sampling the invariant measure of~\eqref{equation:Langevin_epsilon} for different steps~$h$ and parameters~$\varepsilon$, and we compare it with the Euler integrators~\eqref{equation:Euler_EEE} in~$\R^d$ and~\eqref{equation:Euler_Explicit_constrained} on the manifold.
We recall that for dynamics of the form~\eqref{equation:Langevin_epsilon}, the weak convergence implies the convergence for the invariant measure (see~\cite{Talay90eot}).
The numerical experiments of this subsection hint that the uniform accuracy property also extends to the approximation of the invariant measure, but we leave the mathematical analysis for the invariant measure for future work.
We consider the constraint~$\zeta(x)=(x_1^2+x_2^2+x_3^2+R^2-r^2)^2-4R^2(x_1^2+x_2^2)$, with~$R=3$ and~$r=1$, and we choose the map~$f(x)=-25(x_1-R+r,x_2,x_3)$, with~$\sigma=\sqrt{2}$, the test function~$\phi(x)=\abs{x}^2$,~$M=10^7$ trajectories, the final time~$T=10$, and the initial condition~$X_0=(R-r,0,0)$.
Increasing the value of~$T$ does not modify the computed averages, which hints that we reached the equilibrium. The factor~$25$ in~$f$ confines the solution in a reasonably small neighborhood of the torus, which allows a faster convergence to equilibrium and to take fewer trajectories.
We compute the Monte-Carlo estimator~$\widebar{J}=\frac{1}{M}\sum_{m=1}^M \phi(X_N^{(m)}) \simeq \E[\phi(X_N)]$, where~$X_n^{(m)}$ is the~$m$-th realization of the integrator at time~$t_n=nh$, and~$N$ is an integer satisfying~$Nh=T$.
We compare this approximation with a reference value of~$\int_{\R^d} \phi d\mu_\infty^\varepsilon$ computed with the uniformly accurate integrator, by using a timestep~$h_{\text{ref}}=2^{-12}$.
In the case of the constrained Euler scheme~\eqref{equation:Euler_Explicit_constrained}, it amounts to comparing an approximation of~$\int_{\MM} \phi d\mu_\infty^0$ with the reference value of~$\int_{\R^d} \phi d\mu_\infty^\varepsilon$. We observe in Figure~\ref{figure:Plot_torus} that the accuracy of the constrained integrator~\eqref{equation:Euler_Explicit_constrained} for solving the unconstrained problem~\eqref{equation:Langevin_epsilon} deteriorates when~$\varepsilon$ grows larger, as~$\mu_\infty^\varepsilon$ deviates from~$\mu_\infty^0$.
The explicit Euler scheme~\eqref{equation:Euler_EEE} faces stability issues when~$\varepsilon\rightarrow 0$.
The accuracy of the new method for solving~\eqref{equation:Langevin_epsilon} does not deteriorate depending on~$\varepsilon$, and it shares a behavior similar to the constrained Euler scheme~\eqref{equation:Euler_Explicit_constrained} when~$\varepsilon\rightarrow 0$, which is in agreement with Theorem~\ref{theorem:convergence_integrator_epsilon}.
The right graph of Figure~\ref{figure:Plot_torus} shows that the behavior of the error in~$\varepsilon$ is the same for any fixed value of~$h$. This is a numerical confirmation of the uniform accuracy property of the discretization~\eqref{equation:UA_integrator} (in the spirit of the numerical experiments in~\cite{Chartier15uan}), as stated in Theorem~\ref{theorem:uniform_consistency_UA_scheme}.
For any fixed $\varepsilon$, the right graph of Figure~\ref{figure:Plot_torus} also shows that the error decreases when $h\rightarrow 0$. A plot of the error against $h$ for a fixed $\varepsilon$ (not included for conciseness) shows a slope of order one (see Remark \ref{remark:discussion_order_UA_method}).


\begin{figure}[ht]
	\begin{minipage}[c]{.49\linewidth}
		\begin{center}
		\includegraphics[scale=0.5]{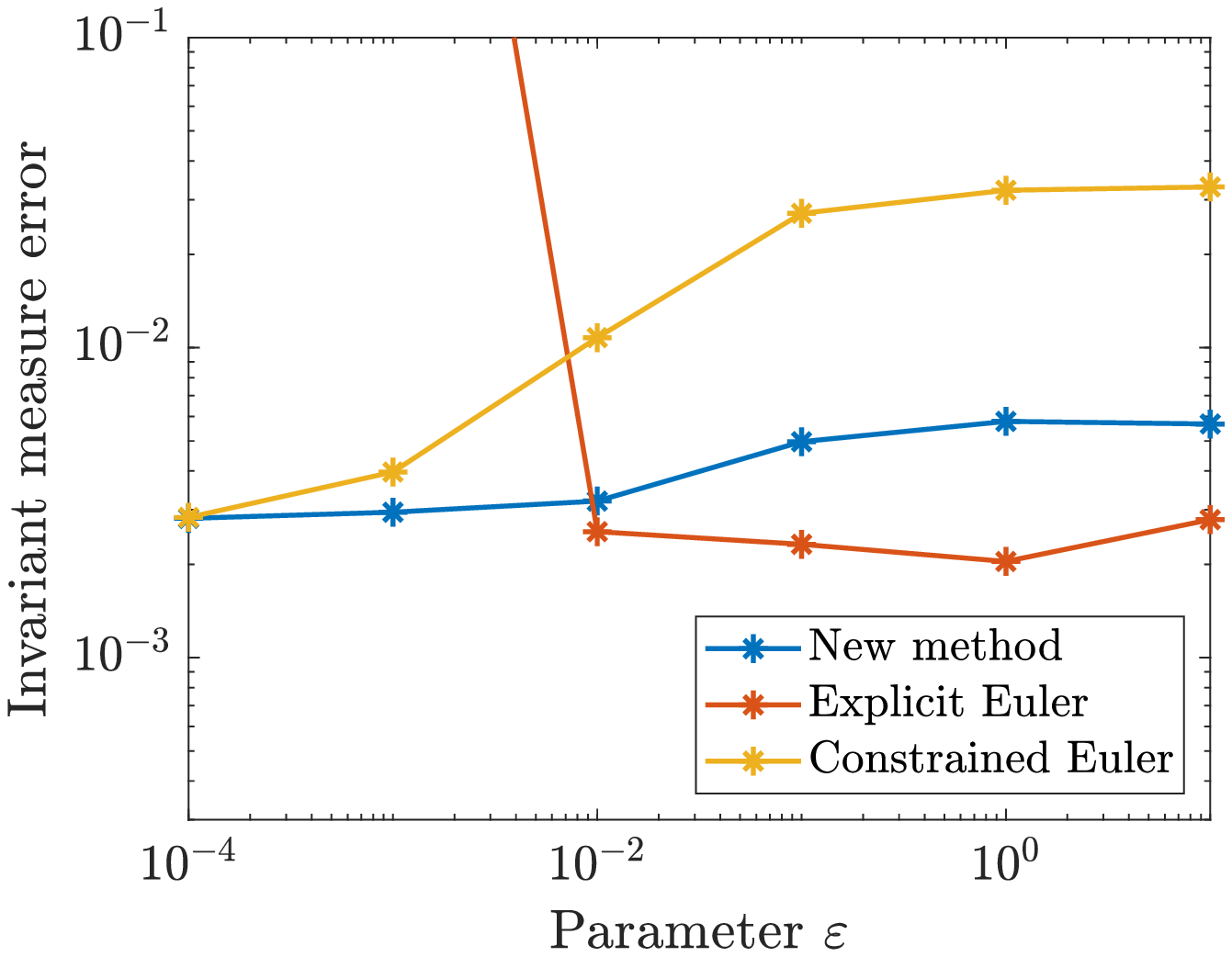}
		\end{center}
	\end{minipage} \hfill
	\begin{minipage}[c]{.49\linewidth}
		\begin{center}
		\includegraphics[scale=0.5]{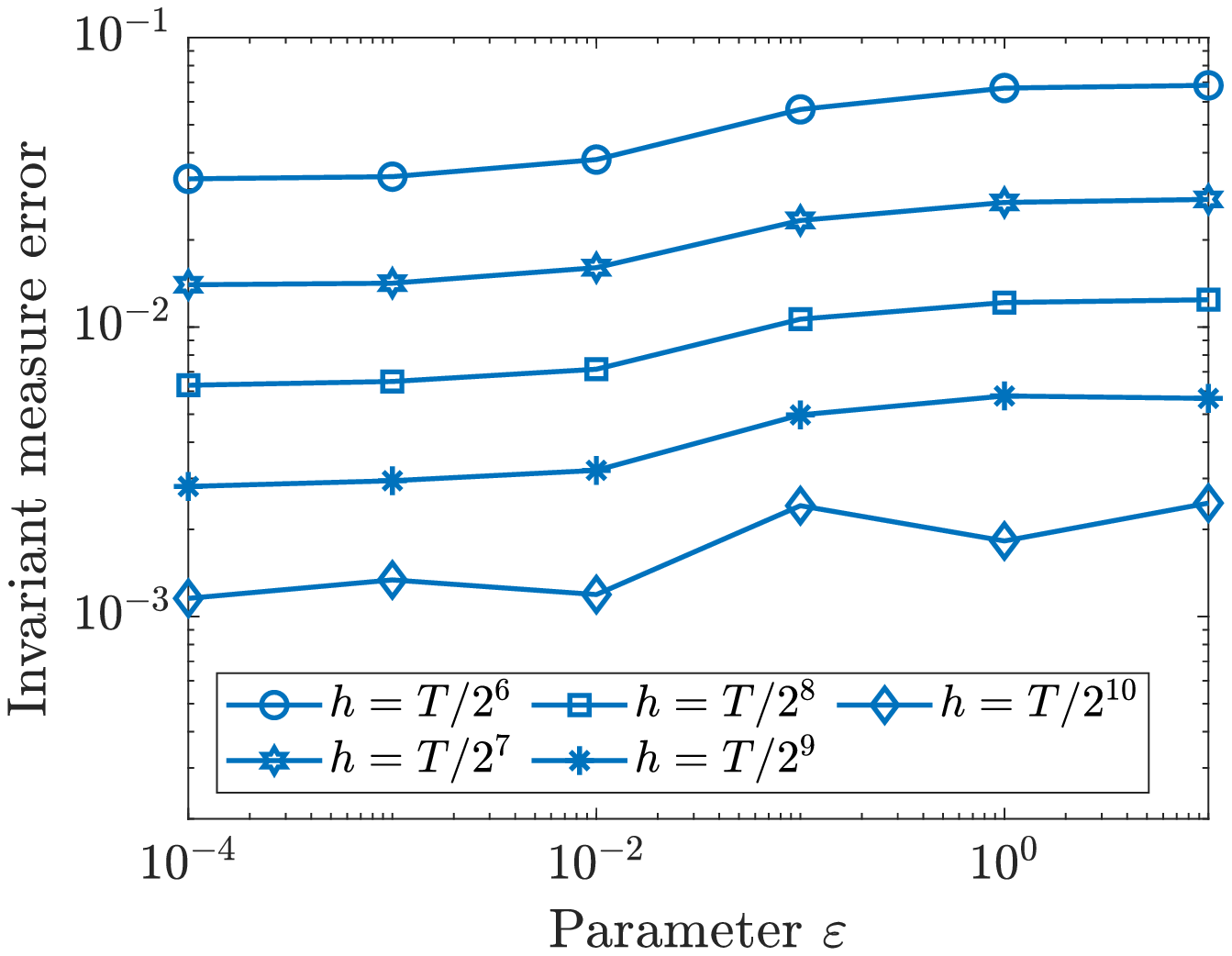}
		\end{center}
	\end{minipage}
	\caption{Error for sampling the invariant measure of penalized Langevin dynamics on a torus in~$\R^3$ of the uniform discretization~\eqref{equation:UA_integrator} and the Euler integrators~\eqref{equation:Euler_EEE} and~\eqref{equation:Euler_Explicit_constrained} for different values of~$\varepsilon$ with~$h=2^{-9}T$ (left), and error curves versus~$\varepsilon$ of the uniformly accurate method for different timesteps~$h=2^{-i}T$ and~$i=6,\dots,10$ (right), with the final time~$T=10$, the maps~$f(x)=-25(x_1-R+r,x_2,x_3)$,~$\phi(x)=\abs{x}^2$, and~$M=10^7$ trajectories.}
	\label{figure:Plot_torus}
\end{figure}

\subsection{Weak approximation on the orthogonal group}

We apply the uniformly accurate method on a compact Lie group (in the spirit of the numerical experiments in~\cite{Zappa18mco,Zhang19eso}) to see how it performs in high dimension and codimension. We choose the orthogonal group~$\Or(m)=\{M\in\R^{m\times m},M^T M=I_m\}$, seen as a submanifold of~$\R^{m^2}$ of codimension~$q=m(m+1)/2$.
We compare the explicit Euler scheme~\eqref{equation:Euler_EEE}, the constrained Euler scheme~\eqref{equation:Euler_Explicit_constrained}, and the new method on~$\MM=\Or(m)$ for~$m=2,\dots,5$ with the parameters~$\varepsilon=0.005$,~$T=1$, and~$h=2^{-7}$.
Note that, as~$h$ and~$\varepsilon$ share the same order of magnitude, the explicit Euler scheme~\eqref{equation:Euler_EEE} can face stability issues, and the solution does not lie on the manifold~$\MM$. Thus, we are in the regime where the convergence results for both Euler schemes~\eqref{equation:Euler_EEE}-\eqref{equation:Euler_Explicit_constrained} do not apply.
We choose~$f=-\nabla V$, where~$V$ is given by
\begin{equation}
\label{equation:potential_O}
V(x)=50\Trace((x-I_{m^2})^T(x-I_{m^2}))
\end{equation}
with the parameters~$\sigma=\sqrt{2}$,~$X_0=I_d$,~$\phi(x)=\Trace(x)$, and~$M=10^6$ trajectories. 
The reference solution for~$J(m)=\E[\phi(X^\varepsilon(T))]$ is computed with the uniformly accurate integrator with~$h_{\text{ref}}=2^{-9}$.
With the factor~$50$ in the potential~\eqref{equation:potential_O}, the trajectories stay close to~$I_{m^2}$, and~$J(m)$ is close to~$\phi(I_{m^2})=m$.
This choice of factor permits one to explore a reasonably small area of~$\Or(m)$, in order to avoid zones close to~$\MM$ where the Gram matrix~$G$ has a bad condition number or is singular, and to reduce the number of trajectories needed.
We observe numerically that replacing the factor~$50$ by~$1$ in~\eqref{equation:potential_O} induces a severe timestep restriction.
We present the results of the experiment in Table~\ref{Table:numerical_results_O}. We omit the results for the explicit Euler scheme~\eqref{equation:Euler_EEE} as the method is inaccurate in this regime (error of size~$1$).
We observe that, in the regime where~$h$ and~$\varepsilon$ share the same order of magnitude, the uniformly accurate integrator performs significantly better than the Euler schemes~\eqref{equation:Euler_Explicit_constrained} and~\eqref{equation:Euler_EEE} for solving the problem~\eqref{equation:Langevin_epsilon}.
In this regime, the Euler method~\eqref{equation:Euler_EEE} faces stability issues, and it is inappropriate to use the constrained Euler scheme~\eqref{equation:Euler_Explicit_constrained} as the solution~$X^\varepsilon(t)$ of~\eqref{equation:Langevin_epsilon} is not close to the solution~$X^0(t)$ of~\eqref{equation:projected_Langevin}.
Moreover, the cost in time of the new method stays the same in average for any value of~$\varepsilon$ (results not included for conciseness).
This confirms numerically the uniform cost of solving the fixed point problem~\eqref{equation:UA_integrator}, as stated in Theorem~\ref{theorem:uniform_fixed_point_problem}.

\begin{table}[H]
	\setcellgapes{3pt}
	\centering
	\begin{tabular}{|c||c|c||c||c|c||c|c|}
	\hline
	$m$ &~$\dim(\MM)$ &~$q$ & \multicolumn{1}{c||}{$J(m)$} &~$\widebar{J}_\text{UA}$ & \multicolumn{1}{c||}{error of~$\widebar{J}_\text{UA}$} &~$\widebar{J}_\text{EC}$ & error of~$\widebar{J}_\text{EC}$\\
	\hhline{|=||=|=||=||=|=||=|=|}
	$2$ & 1 & 3 &~$2.00934$ &~$2.00619$ &~$3.1 \cdot 10^{-3}$ &~$1.99165$ &~$1.8 \cdot 10^{-2}$\\
	\hline
	$3$ & 3 & 6 &~$3.01458$ &~$3.00821$ &~$6.4 \cdot 10^{-3}$ &~$2.97460$ &~$4.0 \cdot 10^{-2}$\\
	\hline
	$4$ & 6 & 10 &~$4.02050$ &~$4.00972$ &~$1.1 \cdot 10^{-2}$ &~$3.94846$ &~$7.2 \cdot 10^{-2}$\\
	\hline
	$5$ & 10 & 15 &~$5.02669$ &~$5.00842$ &~$1.8 \cdot 10^{-2}$ &~$4.91298$ &~$1.1 \cdot 10^{-1}$\\
	\hline
	\end{tabular}
	\caption{Numerical approximation of~$J(m)=\E[\phi(X^\varepsilon(T))]$ for~$2\leq m \leq 5$ with the estimator~$\widebar{J}=M^{-1}\sum_{k=1}^M \phi(X_N^{(k)})$, where~$(X_n)$ is given by the uniform discretization~\eqref{equation:UA_integrator} for~$\widebar{J}_\text{UA}$ and by the constrained Euler scheme~\eqref{equation:Euler_Explicit_constrained} for~$\widebar{J}_\text{EC}$ with their respective errors. The average is taken over~$M=10^6$ trajectories with the potential~\eqref{equation:potential_O},~$\phi(x)=\Trace(x)$, the final time~$T=1$, the stiff parameter~$\varepsilon=0.005$, and the timestep~$h=2^{-7}$.}
	\label{Table:numerical_results_O}
	\setcellgapes{1pt}
\end{table}

\section{Conclusion and future work}
\label{section:future_work}

In this work, we presented a new method for the weak numerical integration of penalized Langevin dynamics evolving in the vicinity of manifolds of any dimension and codimension.
On the contrary of the other existing discretizations, the accuracy of the proposed integrator is independent of the size of the stiff parameter~$\varepsilon$.
Moreover, its cost does not depend on~$\varepsilon$, and it converges to the Euler scheme on the manifold when~$\varepsilon\rightarrow 0$.
Throughout the analysis, we gave an expansion in time of the solution to the penalized Langevin dynamics that is uniform in~$\varepsilon$, as well as new tools for the study of stochastic projection methods for solving stiff SDEs.

Multiple questions arise from the work presented in this paper, with many of great interest for physical applications.
First, it would be interesting to get convergence results with weaker assumptions or to develop uniformly accurate integrators for different penalized dynamics with the same limit when~$\varepsilon \rightarrow 0$ such as the original penalized dynamics~\eqref{equation:modified_Langevin_epsilon} (see~\cite{Ciccotti08pod}).
One could build integrators for penalized dynamics of the form
$$
dX^\varepsilon=f(X^\varepsilon) dt +\sigma dW+\frac{\sigma^2}{4}\nabla\ln(\det(G)) -\frac{1}{\varepsilon}(g G^{-1}\zeta_1)(X^\varepsilon)dt-\frac{1}{\nu}(g G^{-1}\zeta_2)(X^\varepsilon)dt,
$$
where~$\varepsilon$ and~$\nu$ do not share the same order of magnitude, or for constrained dynamics with a penalized term.
One could also build a uniformly accurate numerical scheme with high order in the weak context, or just in the context of the invariant measure (in the spirit of the works~\cite{BouRabee10lra,Leimkuhler13rco,Abdulle14hon,Abdulle15lta,Leimkuhler16tco,Laurent20eab,Laurent21ocf} where numerical schemes of high order for the invariant measure and weak order one were introduced).
Postprocessors~\cite{Vilmart15pif} proved to be an efficient tool for reaching high order for the invariant measure without increasing the cost of the method and could be used in this context.
Moreover, the order conditions presented in~\cite{Laurent20eab} for Runge-Kutta methods for solving Langevin dynamics in~$\R^d$ both in the weak sense and for the invariant measure do not match with the order conditions for solving Langevin dynamics constrained on the manifold~$\MM$, as presented in~\cite{Laurent21ocf}. It would be interesting to create a unified class of high order Runge-Kutta methods with the same order conditions in~$\R^d$, on the manifold~$\MM$ and in the vicinity of~$\MM$.
The discretizations presented in this paper could also be combined with Metropolis-Hastings rejection procedures~\cite{Metropolis53eos,Hastings70mcs}, in the spirit of the works~\cite{Girolami11rml,Brubaker12afo,Lelievre12ldw,Zappa18mco,Lelievre19hmc}, in order to get an exact approximation for the invariant measure with a rejection rate that does not deteriorate in the regime~$\varepsilon\rightarrow 0$.

\bigskip

\noindent \textbf{Acknowledgements.}\
The author would like to thank Gilles Vilmart for helpful discussions.
This work was partially supported by the Swiss National Science Foundation, grants No.\ts 200020\_184614, No.\ts 200021\_162404 and No.\ts 200020\_178752.
The computations were performed at the University of Geneva on the Baobab cluster using the Julia programming language.

\bibliographystyle{abbrv}
\bibliography{Ma_Bibliographie}

\vskip-1ex
\begin{appendices}

\section{Proof of Theorem~\ref{theorem:CV_penalized_constrained}}
\label{section:proof_CV_penalized_constrained}

In this section, we prove the convergence of the penalized Langevin dynamics~\eqref{equation:Langevin_epsilon} to the constrained dynamics~\eqref{equation:projected_Langevin} when~$\varepsilon\rightarrow 0$, as stated in Theorem~\ref{theorem:CV_penalized_constrained}. The proof uses techniques and arguments similar to those in~\cite[Appx.\ts C]{Ciccotti08pod}. However, since we rescaled the stiff term in~\eqref{equation:Langevin_epsilon}, there is no need for a change of time to prove the convergence to the constrained dynamics.

\begin{proof}[Proof of Theorem~\ref{theorem:CV_penalized_constrained}]
In the orthogonal coordinates system given by Assumption~\ref{assumption:regularity_ass}, equation~\eqref{equation:Langevin_epsilon} becomes
\begin{align*}
d\varphi(X^\varepsilon(t))&=(\varphi'f)(X^\varepsilon(t)) dt +\sigma \varphi'(X^\varepsilon(t)) dW(t)\\
&+\frac{\sigma^2}{4}(\varphi' \nabla\ln(\det(G)))(X^\varepsilon(t))dt +\frac{\sigma^2}{2}\Delta \varphi(X^\varepsilon(t)) dt,\\
d\zeta(X^\varepsilon(t))&=(g^T f)(X^\varepsilon(t)) dt +\sigma g^T(X^\varepsilon(t)) dW(t) +\frac{\sigma^2}{4}(g^T \nabla\ln(\det(G)))(X^\varepsilon(t))dt\\
&+\frac{\sigma^2}{2}\Div(g)(X^\varepsilon(t)) dt -\frac{1}{\varepsilon}\zeta(X^\varepsilon(t))dt,
\end{align*}
where we used that~$\varphi'g=0$ and that~$\zeta'=g^T$.
Therefore,~$\zeta(X^\varepsilon(t))$ satisfies
\begin{align*}
\zeta(X^\varepsilon(t))&=\sigma \int_0^t e^{(s-t)/\varepsilon} g^T (X^\varepsilon(s)) dW\\
&+\int_0^t e^{(s-t)/\varepsilon}\Big[g^T f +\frac{\sigma^2}{4}g^T \nabla\ln(\det(G))+\frac{\sigma^2}{2}\Div(g)\Big](X^\varepsilon(s)) dt,
\end{align*}
and a bound on~$\zeta(X^\varepsilon(t))$ is given by
\begin{align}
\label{equation:proof_bound_zeta_X_eps}
\E[\abs{\zeta(X^\varepsilon(t))}^2]&\leq C \int_0^t e^{2(s-t)/\varepsilon} dt+C\Big(\int_0^t e^{(s-t)/\varepsilon}dt\Big)^2
\leq C \varepsilon.
\end{align}

On the other hand, we rewrite~\eqref{equation:projected_Langevin} in the orthogonal coordinates as
\begin{align*}
d\varphi(X^0(t))&=(\varphi'f)(X^0(t)) dt +\sigma \varphi'(X^0(t)) dW\\
&+\frac{\sigma^2}{4}(\varphi' \nabla\ln(\det(G)))(X^0(t))dt+\frac{\sigma^2}{2}\Delta \varphi(X^0(t)) dt,\\
d\zeta(X^0(t))&=0,
\end{align*}
where we used that
$$\frac{1}{2}\varphi' \nabla\ln(\det(G)) +\Delta \varphi=\sum_{i=1}^d \varphi'(\Pi_\MM'(\Pi_\MM e_i) e_i) +\sum_{i=1}^d \varphi''(\Pi_\MM e_i,\Pi_\MM e_i).$$
With Assumption~\ref{assumption:regularity_ass}, we obtain
$$\E[\abs{\varphi(X^\varepsilon(t))-\varphi(X^0(t))}^2]\leq C \int_0^t \E[\abs{X^\varepsilon(s)-X^0(s)}^2]ds,$$
and from~\eqref{equation:proof_bound_zeta_X_eps} we deduce that
$$\E[\abs{\zeta(X^\varepsilon(t))-\zeta(X^0(t))}^2]\leq C \varepsilon.$$
Thus,~$\psi(X^\varepsilon(t))$ satisfies
$$\E[\abs{\psi(X^\varepsilon(t))-\psi(X^0(t))}^2]\leq C \int_0^t \E[\abs{X^\varepsilon(s)-X^0(s)}^2]ds +C\varepsilon,$$
and, as~$\psi^{-1}$ is Lipschitz, we have
$$\E[\abs{X^\varepsilon(t)-X^0(t)}^2]\leq C\int_0^t \E[\abs{X^\varepsilon(s)-X^0(s)}^2]ds +C\varepsilon,$$
which, with the use of the Gronwall lemma, gives the desired estimate.
\end{proof}

\end{appendices}

\end{document}